\documentclass[12pt, final]{article}
\usepackage{amsmath, amssymb, amsfonts, amsxtra, amscd}
\usepackage{amsthm}
\usepackage{latexsym}
\usepackage{graphicx}
\usepackage{verbatim}
\usepackage{mathrsfs}
\usepackage{algorithmic}
\usepackage{float}
\usepackage{url}
\usepackage[lofdepth, lotdepth]{subfig}
\usepackage[top=2.5cm, bottom=3cm, left=2.5cm, right=2.5cm]{geometry}

\theoremstyle{plain}
  \newtheorem{theorem}{Theorem}[section]
  \newtheorem{lemma}[theorem]{Lemma}

\theoremstyle{definition}
  \newtheorem{definition}[theorem]{Definition}
  
  \newtheorem{example}[theorem]{Example}
  \newtheorem{remark}[theorem]{Remark}

\newcommand{\VT}{\mathcal{V}(T)}
\newcommand{\ET}{\mathcal{E}(T)}
\newcommand{\TT}{\mathscr{T}}
\newcommand{\des}{\mathsf{des}}
\newcommand{\mdc}{\mathsf{mdc}}
\newcommand{\mdct}{\mathsf{mdc}(\mathscr{T})}
\newcommand{\K}{{\mathcal K}}
\newcommand{\PP}{{\mathscr P}}
\newcommand{\FF}{{\mathscr F}}
\newcommand{\s}{\mathsf{s}}

\newcommand{\mrq}{\mathsf{minrk}_q}

\newcommand{\A}{{\mathcal A}}
\newcommand{\B}{{\mathcal B}}
\newcommand{\C}{{\mathcal C}}

\newcommand{\E}{{\mathcal E}}

\newcommand{\G}{{\mathcal G}}
\newcommand{\HH}{{\mathcal H}}
\newcommand{\I}{{\mathcal I}}
\newcommand{\J}{{\mathcal J}}

\newcommand{\Q}{{\mathcal Q}}
\newcommand{\X}{{\mathcal X}}

\newcommand{\V}{{\mathcal V}}

\newcommand{\N}{{\mathcal N}}

\newcommand{\sS}{{\mathcal S}}

\newcommand{\cL}{\mathscr{L}}

\newcommand{\bM}{{\boldsymbol{M}}}

\newcommand{\rank}{{\mathsf{rank}_q}}

\newcommand{\define}{\stackrel{\mbox{\tiny $\triangle$}}{=}}

\newcommand{\bx}{{\boldsymbol{x}}}

\newcommand{\be}{{\boldsymbol e}}

\newcommand{\aG}{{\alpha(\G)}}

\newcommand{\Gc}{\overline{{\mathcal{G}}}}
\newcommand{\VG}{\mathcal{V}(\mathcal{G})}
\newcommand{\VGc}{\mathcal{V}(\overline{\mathcal{G}})}
\newcommand{\EG}{\mathcal{E}(\mathcal{G})}
\newcommand{\EGc}{\mathcal{E}(\overline{\mathcal{G}})}

\newcommand{\kp}{\kappa}

\newcommand{\Ga}{\Gamma}

\newcommand{\nin}{\noindent}

\newcommand{\ft}{\mathbb{F}_2}
\newcommand{\fq}{\mathbb{F}_q}

\newcommand{\et}{{\emph{et al.}}}

\newcommand{\fkE}{{\mathfrak E}}

\begin{document}

\title{Polynomial Time Algorithm for Min-Ranks of Graphs with Simple Tree Structures{\small $^\dagger$}}



\author{Son Hoang Dau{\small $^\ddagger$} \and Yeow Meng Chee{\small $^\ddagger$}}
\date{}




\maketitle
\let\thefootnote\relax\footnotetext{$^\dagger$This work is supported in part by the National Research Foundation of Singapore (Research Grant NRF-CRP2-2007-03).}
\let\thefootnote\relax\footnotetext{$^\ddagger$Division of Mathematical Sciences, School of Physical and Mathematical Sciences,
               Nanyang Technological University, \it{dausonhoang84@gmail.com, ymchee@ntu.edu.sg}}

\begin{abstract}
The \emph{min-rank} of a graph was introduced by Haemers (1978) to bound the Shannon capacity of a graph. This parameter of a graph has recently gained much more attention from the research community after the work of Bar-Yossef {\et} (2006). In their paper, it was shown that the min-rank of a graph $\G$ characterizes the optimal scalar linear solution of an instance of the Index Coding with Side Information (ICSI) problem described by the graph $\G$. 

It was shown by Peeters (1996) that
computing the min-rank of a general graph is an NP-hard problem. 
There are very few known families of graphs whose min-ranks can be found in 
polynomial time.   
In this work, we introduce a new family of graphs with efficiently computed min-ranks.
Specifically, we establish a polynomial time dynamic programming algorithm 
to compute the min-ranks of graphs having \emph{simple tree structures}. 
Intuitively, such graphs are obtained by gluing together, in a tree-like structure, 
any set of graphs for which the min-ranks can be determined in polynomial time.  
A polynomial time algorithm to recognize such graphs is also proposed. 
\end{abstract}

\section{Introduction}

\subsection{Background}

Building communication schemes which allow participants to communicate efficiently
has always been a challenging yet intriguing problem for information theorists. 
Index Coding with Side Information (ICSI) (\cite{BirkKol98}, \cite{BirkKol2006}) is a 
communication scheme dealing with broadcast channels in which receivers have prior side information 
about the messages to be transmitted. Exploiting the knowledge about the side information,
the sender may significantly reduce the number of required transmissions compared
with the naive approach (see Example~\ref{ex:mr-ex}). As a consequence, the efficiency of the communication over 
this type of broadcast channels could be dramatically improved. 
Apart from being a special case of the well-known (non-multicast) Network Coding problem
(\cite{Ahlswede}, \cite{KoetterMedard2003}), the ICSI
problem has also found various potential applications on its own, such as audio- and 
video-on-demand, daily newspaper delivery, data pushing, and opportunistic wireless networks
(\cite{BirkKol98}, \cite{BirkKol2006}, \cite{Yossef}, \cite{Rouayheb2009}, \cite{Katti2006}, \cite{Katti2008}). 

In the work of Bar-Yossef {\et} \cite{Yossef}, the optimal transmission rate of scalar linear index codes for an ICSI instance
was neatly characterized by the so-called \emph{min-rank} of the side information graph
corresponding to that instance. 
The concept of min-rank of a graph was first introduced by Haemers \cite{Haemers1978}, which serves as an upper bound for the celebrated Shannon capacity of a graph \cite{Shannon1956}. This upper bound, as pointed out by Haemers, 
although is usually not as good as the Lov\'asz bound \cite{Lovasz1979}, is sometimes tighter and easier to compute.
However, as shown by Peeters \cite{Peeters96}, computing the min-rank of a general graph 
(that is, the Min-Rank problem) is a hard task. 
More specifically, Peeters showed that deciding whether the min-rank of a graph 
is smaller than or equal to three is an NP-complete problem. 
The interest in the Min-Rank problem has grown significantly after the work of Bar-Yossef {\et} \cite{Yossef}. 
Subsequently, Lubetzky and Stav \cite{LubetzkyStav} constructed a family of graphs for which the min-rank over the binary field is strictly larger than the min-rank over a nonbinary field. This disproved a conjecture by Bar-Yossef {\et}~\cite{Yossef} which stated that binary min-rank provides an optimal solution for the ICSI problem. 
Exact and heuristic algorithms to find min-rank over the binary field of a graph was developed in the work of Chaudhry and Sprintson \cite{ChaudhrySprintson}. The min-rank of a random graph was investigated by Haviv and Langberg \cite{HavivLangberg2011}. A dynamic programming approach was proposed by Berliner and Langberg~\cite{BerlinerLangberg2011} to compute in polynomial time min-ranks of outerplanar graphs. 
Algorithms to approximate min-ranks of graphs with bounded
min-ranks were studied by Chlamtac and Haviv~\cite{ChlamtacHaviv2011}. 
They also pointed out a tight upper bound for the Lov\'asz $\vartheta$-function \cite{Lovasz1979}
of graphs in terms of their min-ranks.  
It is also worth noting that approximating min-ranks of graphs within any constant ratio 
is known to be NP-hard (see Langberg and Sprintson~\cite{LangbergSprintson2008}).

\subsection{Our Contribution}
So far, families of graphs whose min-ranks are either known or computable 
in polynomial time are the following: odd cycles and their complements, 
perfect graphs, and outerplanar graphs.
Inspired by the work of Berliner and Langberg~\cite{BerlinerLangberg2011}, 
we develop a dynamic programming algorithm 
to compute the min-ranks of graphs having \emph{simple tree structures}. 
Loosely speaking, such a graph can be described as a compound rooted tree, 
the nodes of which are induced subgraphs whose min-ranks can be computed in polynomial time. 
\begin{figure}[H]
\centering
		\includegraphics{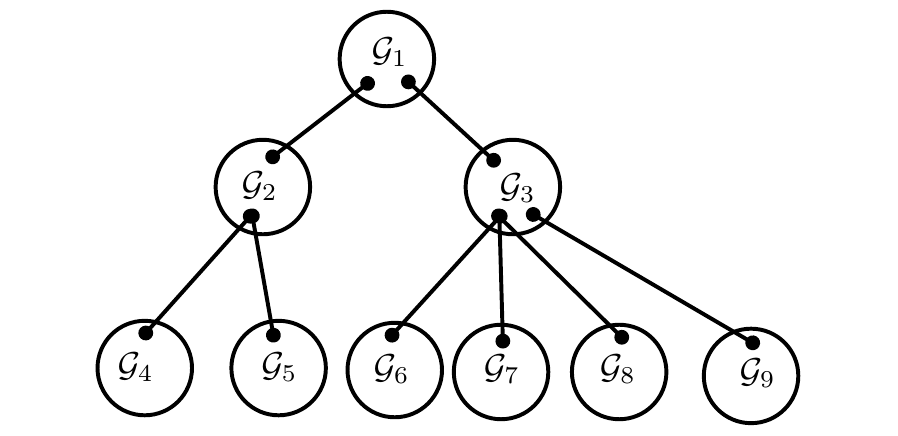}
\caption{A graph $\G$ with a simple tree structure}
\label{fig:graph_simple_tree_structure}
\end{figure}
As an illustrative example, a graph $\G$ with a simple tree structure is depicted
in Figure~\ref{fig:graph_simple_tree_structure}. 
In this example, each induced subgraph (node) $\G_i$ ($i \in [9]$) of $\G$ is either 
a perfect graph or an outerplanar graph (hence $\G_i$'s min-rank can be efficiently computed). 
The dynamic programming algorithm (Algorithm~1) computes 
the min-ranks of the subtrees, from the leaves to the root, 
in a bottom-up manner. The task of computing the min-rank of a graph is accomplished 
when the computation reaches the root of the compound tree.
Let $\FF_\PP(c)$, roughly speaking, denote the family of graphs with simple tree structures 
where each node in the tree structure is connected 
to its child nodes via at most $c$ vertices.
For instance, the graph $\G$ depicted in Figure~\ref{fig:graph_simple_tree_structure} 
belongs to the family $\FF_\PP(2)$.
We prove that Algorithm~1 runs in polynomial time if $\G \in \FF_\PP(c)$, 
and also provide another algorithm (Algorithm~2) that recognizes 
a member of $\FF_\PP(c)$ in polynomial time, for any constant $c > 0$. 

In fact, Algorithm~1 still runs in polynomial time for graphs belonging to a larger family $\FF_\PP(c\log(\cdot))$. This family consists 
of graphs $\G$ with simple tree structures where each node in
the tree structure is connected to its child nodes via at most
$c\log |\VG|$ vertices. However, finding a
polynomial time recognition algorithm for members of
$\FF_\PP(c\log(\cdot))$ is still an open problem.

Another way to look at our result is as follows. 
From a given set of graphs $\G_i$ ($i \in [k]$) whose min-ranks 
can be computed in polynomial time, one can build a new graph
$\G$ such that $\G_i$ ($i \in [k]$) are all the connected components
of $\G$. Then by Lemma~\ref{lem:connected_components}, 
the min-rank of $\G$ can be trivially computed by taking the sum
of all the min-ranks of $\G_i$ ($i \in [k]$). 
This is a \emph{trivial} way to build up a new graph whose min-rank 
can be efficiently computed from a given set of graphs whose 
min-ranks can be efficiently computed. 
Our main contribution is to provide a method to build up in 
a \emph{nontrivial} way an infinite family of 
\emph{new} graphs with min-ranks computable in polynomial time 
from \emph{given} families of graphs with min-ranks computable in polynomial time.
This new family can be further enlarged whenever a new family of graphs 
(closed under induced subgraphs) with min-ranks computable in polynomial time is discovered. Using this method, roughly speaking, from a given set of graphs,
we build up a new one by introducing edges that connect these graphs 
in such a way that a tree structure is formed.  
  
It is also worth mentioning that 
the min-ranks of \emph{all} non-isomorphic graphs of order up to $10$ can be found
using a computer program that combines a SAT-based approach \cite{ChaudhrySprintson} 
and a Branch-and-Bound approach.

\subsection{Organization}
The paper is organized as follows. 
Basic notation and definitions are presented in Section~\ref{sec:not_def}. 
The ICSI problem is formally formulated in Section~\ref{sec:ICSI_formulation}. 
The dynamic programming algorithm that computes in polynomial time min-ranks 
of the graphs with simple tree structures is presented in Section~\ref{sec:type_I}. 
An algorithm that recognizes such graphs in polynomial time is also developed therein. 
We mention the computation of min-ranks of all non-isomorphic graphs of small
orders in Section~\ref{sec:small_graphs}. 
Finally, some interesting open problems are proposed in Section~\ref{sec:open_problems}. 

\section{Notation and Definitions}
\label{sec:not_def}

We use $[n]$ to denote the set of integers $\{1,2,\ldots,n\}$.
We also use $\fq$ to denote the finite field of $q$ elements. 
For an $n \times k$ matrix $\bM$, let $\bM_i$ denote the $i$th row of $\bM$. 
For a set $E \subseteq [n]$, let $\bM_E$ denote the $|E| \times k$ 
sub-matrix of $\bM$ formed by rows of $\bM$ that are indexed by the elements of $E$.
For any matrix $\bM$ over $\fq$, 
we denote by $\rank(\bM)$ the rank of $\bM$ over $\fq$. 

A simple \emph{graph} is a pair $\G = (\VG, \EG)$ where $\VG$ is the set of vertices of $\G$\index{vertex}
and $\EG$ is a set of \emph{unordered} pairs of distinct vertices of $\G$. We refer to $\EG$ as the set of \emph{edges} of $\G$. A typical edge of $\G$ is of the form $\{u,v\}$ where $u\in \VG$, $v \in \VG$, and $u \neq v$.
If $e = \{u,v\} \in \EG$ we say that $u$ and $v$ are adjacent. 
We also refer to $u$ and $v$ as the \emph{endpoints} of $e$.
We denote by $N^\G(u)$ the set of neighbors of $u$, namely, the set of vertices
adjacent to $u$. 

Simple graphs have no loops and no parallel edges. 
In the scope of this paper, only simple graphs are considered. 
Therefore, we use graphs to refer to simple graphs for succinctness.  
The number of vertices $|\VG|$ is called the \emph{order} of $\G$,
whereas the number of edges $|\EG|$ is called the \emph{size} of $\G$.
The \emph{complement} of a graph $\G = (\VG, \EG)$, denoted by $\Gc = (\VGc, \EGc)$, is defined as follows.
The vertex set $\VGc = \VG$. The arc set
\[
\EGc = \big\{\{u,v\}: \ u, v \in \VG, \ u \neq v, \ \{u,v\} \notin \EG \big\}.
\] 

A \emph{subgraph} of a graph $\G$ is a graph whose vertex set $V$ is a subset of that of $\G$ 
and whose edge set is a subset of that of $\G$ restricted on the vertices in $V$.
The subgraph of $\G$ \emph{induced} by $V \subseteq \VG$ is a graph whose vertex set is $V$, 
and edge set is $\{\{u,v\}: \ u \in V,\ v \in V, \ \{u,v\} \in \EG\}$. 
We refer to such a graph as an \emph{induced subgraph} of $\G$.

A \emph{path} in a graph $\G$ is a sequence of pairwise distinct vertices 
$(u_1,u_2,\ldots,u_\ell)$, such that $\{u_i,u_{i+1}\} \in \EG$ for all $i \in [\ell-1]$. 
A \emph{cycle} is a path $(u_1,u_2,\ldots,u_\ell)$ ($\ell \geq 3$) such that $u_1$ and $u_\ell$
are also adjacent.  
A graph is called \emph{acyclic} if it contains no cycles.

A graph is called \emph{connected} if there is a path from each vertex in the graph to every other vertex. 
The \emph{connected components} of a graph are its maximal connected subgraphs.
A \emph{bridge} is an edge whose deletion increases
the number of connected components. 
In particular, an edge in a connected graph is a bridge 
if and only if its removal renders the graph disconnected.  

A collection of subsets $V_1, V_2, \ldots, V_k$ of a set $V$ is said to 
\emph{partition} $V$ if $\cup_{i = 1}^k V_i = V$ and $V_i \cap V_j = \varnothing$
for every $i \neq j$. In that case, $[V_1, V_2, \ldots, V_k]$ is referred
to as a partition of $V$, and $V_i$'s ($i \in [k]$) are called \emph{parts}
of the partition. 

A \emph{tree} is a connected acyclic graph.
A \emph{rooted tree} is a tree with one special vertex designated to be the \emph{root}. 
In a rooted tree, there is a unique path that connects the root to each other vertex. 
The \emph{parent} of a vertex $v$ is the vertex
connected to it on the path from $v$ to the root. Every vertex except the root
has a unique parent. If $v$ is the parent of a vertex $u$ then $u$
is the \emph{child} of $v$.
An \emph{ancestor} of $v$ is a vertex lying on the path connecting $v$ to the root. 
If $w$ is an ancestor of $v$, then $v$ is a \emph{descendant} of $w$. We use $\des_T(w)$ to denote the set of descendants of $w$ in
a rooted tree $T$. 

A graph $\G$ is called \emph{outerplanar} (Chartrand and Harary~\cite{ChartrandHarary1967}) if it can be
drawn in the plane without crossings in such a way that
all of the vertices belong to the unbounded face of the drawing.

An \emph{independent set} in a graph $\G$ is a set of vertices of $\G$ 
with no edges connecting any two of them. 
The cardinality of a largest independent set in $\G$ is referred to as 
the \emph{independence number} of $\G$, denoted by $\aG$. 
The \emph{chromatic number} of a graph $\G$ is the smallest
number of colors $\chi(\G)$ needed to color
the vertices of $\G$ so that no two adjacent vertices share
the same color. 

A graph $\G$ is called \emph{perfect} if for every induced subgraph $\HH$ of $\G$, it holds that $\alpha(\HH) =\chi(\overline{\HH})$. 
Perfect graphs include families of graphs such as trees, bipartite graphs, interval graphs, and chordal graphs.
For the full characterization of perfect graphs, the reader can refer to~\cite{Chudnovsky}.

\section{The Index Coding with Side Information Problem}
\label{sec:ICSI_formulation}

The ICSI problem is formulated as follows. 
Suppose a sender $S$ wants to send a vector $\bx = (x_1,x_2,\ldots,x_n)$, where $x_i \in \fq$ for all $i\in [n]$, to $n$ receiver $R_1,R_2,\ldots, R_n$. Each $R_i$ possesses some prior \emph{side information}, consisting of the messages $x_j$'s, $j \in \X_i \subsetneq [n]$, and is interested in receiving a single message $x_i$. The sender $S$ broadcasts a codeword $\fkE(\bx) \in \fq^\kp$ that enables each receiver $R_i$ to recover $x_i$ based on its side information. Such a mapping $\fkE$ is called an \emph{index code} over $\fq$. We refer to $\kp$ as the \emph{length} of the index code. The objective of $S$ is to find an \emph{optimal} index code, that is, an index code which has minimum length. The index code is called \emph{linear} if $\fkE$ is a linear mapping.

If it is required that $x_j \in \X_i$ if and only if $x_i \in \X_j$ for every $i \neq j$, then the ICSI instance is called \emph{symmetric}. 
Each symmetric instance of the ICSI problem can be described by the so-called side information graph \cite{Yossef}. Given $n$ and $\X_i$, $i \in [n]$, the \emph{side information graph} $\G = (\VG, \EG)$ is defined as follows. The vertex set $\VG = \{u_1,u_2,\ldots,u_n\}$. The edge set $\EG = \cup_{i \in [n]} \big\{\{u_i,u_j\}:\ j \in \X_i \big\}$. 
Sometimes we simply take $\VG = [n]$ and $\EG = \cup_{i \in [n]} \big\{\{i,j\}:\ j \in \X_i \big\}$.

\vskip 10pt 
\begin{definition}[\cite{Haemers1978}]
\label{def:mr-def}
Let $\G = \big(\VG=\{u_1,u_2,\ldots,u_n\},\EG\big)$ be a graph of order $n$.  
\begin{enumerate}
	\item A matrix $\bM=(m_{u_i,u_j}) \in \fq^{n \times n}$ (whose rows
	and columns are labeled by the elements of $\VG$) is said to \emph{fit} $\G$ if 
\[
\begin{cases} m_{u_i,u_j} \neq 0,& i = j, \\ m_{u_i,u_j} = 0,& i \neq j, \ \{u_i,u_j\} \notin \EG. \end{cases}
\]
  \item The \emph{min-rank} of $\G$ over $\fq$ is defined to be
	\[
	\mrq(\G) \define \min\left\{\rank(\bM): \ \bM \in \fq^{n \times n} \text{ and } \bM \text{ fits } \G \right\}.
	\]
\end{enumerate}
\end{definition} 
\vskip 10pt

\begin{theorem}[\cite{Yossef, LubetzkyStav}]
\label{thm:mr_theorem}
The length of an optimal linear index code over $\fq$ for the ICSI instance described by $\G$ is $\mrq(\G)$.  
\end{theorem}
\vskip 10pt

\begin{example}
\label{ex:mr-ex}
Consider an ICSI instance with $n = 5$ and $\X_1 = \{2,3,5\}$, $\X_2 = \{1,3\}$, $\X_3 = \{1,2,4\}$, $\X_4 = \{3,5\}$, and 
$\X_5 = \{1,4\}$ (Figure~\ref{fig:ic-ex}). 
\begin{figure}[H]
\centering
\subfloat[An ICSI instance]{
\includegraphics{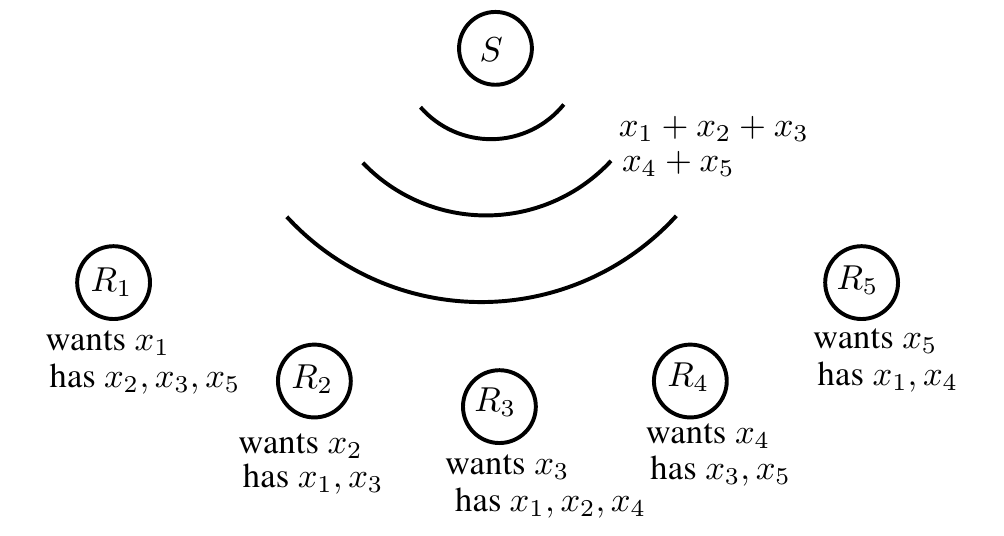}
\label{fig:ic-ex}
}
\qquad
\subfloat[Graph $\G$]{
\includegraphics{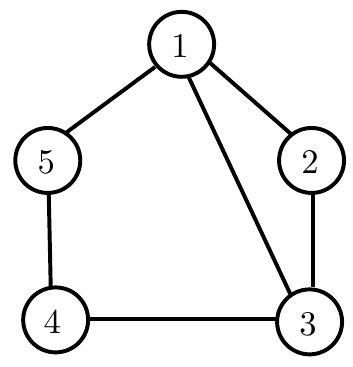}
\label{fig:sub-fig1}
}
\caption{An ICSI instance and the side information graph}
\end{figure}

The side information graph $\G$ that describes this instance is depicted in Figure~\ref{fig:sub-fig1}. 
A matrix fitting $\G$ of rank two over $\ft$, which is the minimum rank, is shown in Figure~\ref{fig:sub-fig3}. 
By Theorem~\ref{thm:mr_theorem}, an optimal linear index code over $\ft$ for this instance 
has length two. In other words, using linear index codes over $\ft$, the smallest number of transmissions
required is two. 
The sender can broadcast two packets $x_1 + x_2 + x_3$ and $x_4+x_5$. 
The decoding process goes as follows. 
Since $R_1$ already knows $x_2$ and $x_3$, it obtains $x_1$ by adding
$x_2$ and $x_3$ to the first packet: 
$x_1 = x_2 + x_3 + (x_1+ x_2+x_3)$.
Similarly, $R_2$ obtains $x_2 = x_1 + x_3 + (x_1+ x_2+ x_3)$; $R_3$ obtains 
$x_3 = x_1 + x_2 + (x_1 + x_2 + x_3)$; $R_4$ obtains 
$x_4 = x_5 + (x_4 + x_5)$; $R_5$ obtains $x_5 = x_4 + (x_4 + x_5)$.
This index code \emph{saves} three transmissions, compared with the trivial solution
when the sender simply broadcasts five messages $x_1$, $x_2$, $x_3$, $x_4$, and $x_5$. 

\vskip 10pt 
\begin{figure}[h]
\centering
\subfloat[A matrix of rank three that fits $\G$]{
$
\bM^{(1)} = 
\begin{pmatrix}
1 & 1 & 0 & 0 & 0 \\
1 & 1 & 0 & 0 & 0 \\
0 & 0 & 1 & 1 & 0 \\
0 & 0 & 1 & 1 & 0 \\
1 & 0 & 0 & 0 & 1
\end{pmatrix}
$
\label{fig:sub-fig2}
}
\qquad \qquad
\subfloat[A matrix of rank two (minimum rank) that fits $\G$]{
$
\bM^{(2)} = 
\begin{pmatrix}
1 & 1 & 1 & 0 & 0 \\
1 & 1 & 1 & 0 & 0 \\
1 & 1 & 1 & 0 & 0 \\
0 & 0 & 0 & 1 & 1 \\
0 & 0 & 0 & 1 & 1
\end{pmatrix}
$
\label{fig:sub-fig3}
}
\caption{Examples of matrices fitting $\G$}
\label{fig:example_mr}
\end{figure}
We may observe that the index code above encodes $\bx$
by taking the dot products of $\bx$ and the first and the forth
rows of the matrix $\bM^{(2)}$ (Figure~\ref{fig:sub-fig3}).  
These two rows, in fact, form a basis of the row space of
this matrix. Therefore, this index code has length equal to 
the rank of $\bM^{(2)}$, which is two. This argument partly explains why the shortest length of a linear index code over $\fq$ for the ICSI instance described by $\G$ is equal to the minimum rank
of a matrix fitting $\G$ (Theorem~\ref{thm:mr_theorem}).   
\end{example}

\begin{lemma}[Folklore]
\label{lem:connected_components}
Let $\G = (\VG, \EG)$ be a graph. 
Suppose that $\G_1, \G_2, \ldots, \G_k$ are subgraphs 
of $\G$ that satisfy the following conditions:
\begin{enumerate}
\item
The sets $\V(\G_i)$'s, $i \in [k]$, partition $\V(\G)$;
\item 
There is no edge of the form $\{u,v\}$ where $u \in \V(\G_i)$ and
$v \in \V(\G_j)$ for $i \neq j$. 
\end{enumerate}
Then
\[
\mrq(\G) = \sum_{i = 1}^k \mrq(\G_i).
\]
In particular, the above equality holds if $\G_1, \G_2, \ldots, \G_k$ are all connected components of $\G$.
\end{lemma}
\begin{proof}
The proof follows directly from the fact that a matrix fits $\G$ 
if and only if it is a block diagonal matrix (relabeling the vertices if necessary) 
and the block sub-matrices fit the corresponding subgraphs $\G_i$'s, $i \in [k]$.  
Note also that the rank of a block diagonal matrix is equal to the sum of the ranks 
of its block sub-matrices.  
\end{proof} 

This lemma suggests that it is often sufficient to study the min-ranks of 
graphs that are \emph{connected}.

\section{On Min-Ranks of Graphs with Simple Tree Structures}
\label{sec:type_I}

We present in this section a new family of graphs whose min-ranks
can be found in polynomial time. 

\subsection{Simple Tree Structures}
\label{subsec:tree_structure}

We denote by $\PP$ an arbitrary collection of finitely many families of graphs that satisfy the following properties:
\begin{enumerate}
  \item[(P1)] Each family is closed under the operation of taking induced subgraphs, that is, every 
	induced subgraph of a member of a family in $\PP$ also belongs to that family;
	\item[(P2)] There is a polynomial time algorithm to recognize a member of each family;
	\item[(P3)] There is a polynomial time algorithm to find the min-rank of every member of each family. 
\end{enumerate}

For instance, we may choose such a $\PP$ to be the collection of the following three families: perfect graphs \cite{Yossef}, \cite{CornuejolsLiuVuskovic2003}, outerplanar graphs \cite{BerlinerLangberg2011}, \cite{Wiegers1987}, and graphs of orders bounded by a constant. Instead of saying that a graph $\G$ belongs to a family in $\PP$, with a slight abuse of notation, we 
often simply say that $\G \in \PP$. Note that if $\G \in \PP$ then 
the min-rank of any of its induced subgraph can also be found in polynomial time.  
 
Let $U$ and $V$ be two disjoint nonempty sets of vertices of $\G$. \index{$\s_\G(U,V)$}
Let 
\[
\s_\G(U,V) = \big|\big\{\{u,v\}:\ u \in U,\ v \in V, \ \{u,v\} \in \EG \big\}\big|,
\]
denotes the number of edges each of which has one endpoint in $U$ and the other endpoint in~$V$.  

\vskip 10pt 
\begin{definition}    
Let $\PP$ be a collection of finitely many families of graphs that satisfy (P1), (P2), and (P3).   
A connected graph $\G = (\VG, \EG)$ is said to have a ($\PP$) \emph{simple tree structure} 
if there exists a partition $\Ga = [\V_1,\V_2,\ldots,\V_k]$ of the 
vertex set $\VG$ that satisfies the following three requirements:

\begin{enumerate}
	\item[(R1)] The $\V_i$-induced subgraph $\G_i$ of $\G$ belongs to a family in $\PP$, for every $i \in [k]$;
	\item[(R2)] $\s_\G(\V_i,\V_j) \in \{0,1\}$ for every $i \neq j$; 
	\item[(R3)] The graph $T = (\VT, \ET)$, where $\VT = [k]$ and 
	\[
	\ET = \big\{ \{i,j\}:\ \s_\G(\V_i,\V_j) = 1 \big\},
	\]
	is a rooted tree; The tree $T$ can also be thought of as a graph 
	obtained from $\G$ by contracting each $\V_i$ to a single vertex. 
\end{enumerate}
The 2-tuple $\TT= (\Ga, T)$ is called a ($\PP$) \emph{simple tree structure} of $\G$.
\end{definition}

\begin{example}
\label{ex:tree-structure}
Suppose the $\V_i$-induced subgraph $\G_i$ of $\G$ is either 
a perfect graph or an outerplanar graph for every $i \in [9]$ . 
Let $\PP$ consist of the families of perfect graphs and outerplanar graphs. 
Then $\TT = ([\V_1, \V_2, \ldots,\V_9], T)$ is a ($\PP$) simple tree structure of $\G$
where $T$ is depicted in Figure~\ref{fig:tree_structure}. 
\begin{figure}[H]
\centering
\includegraphics{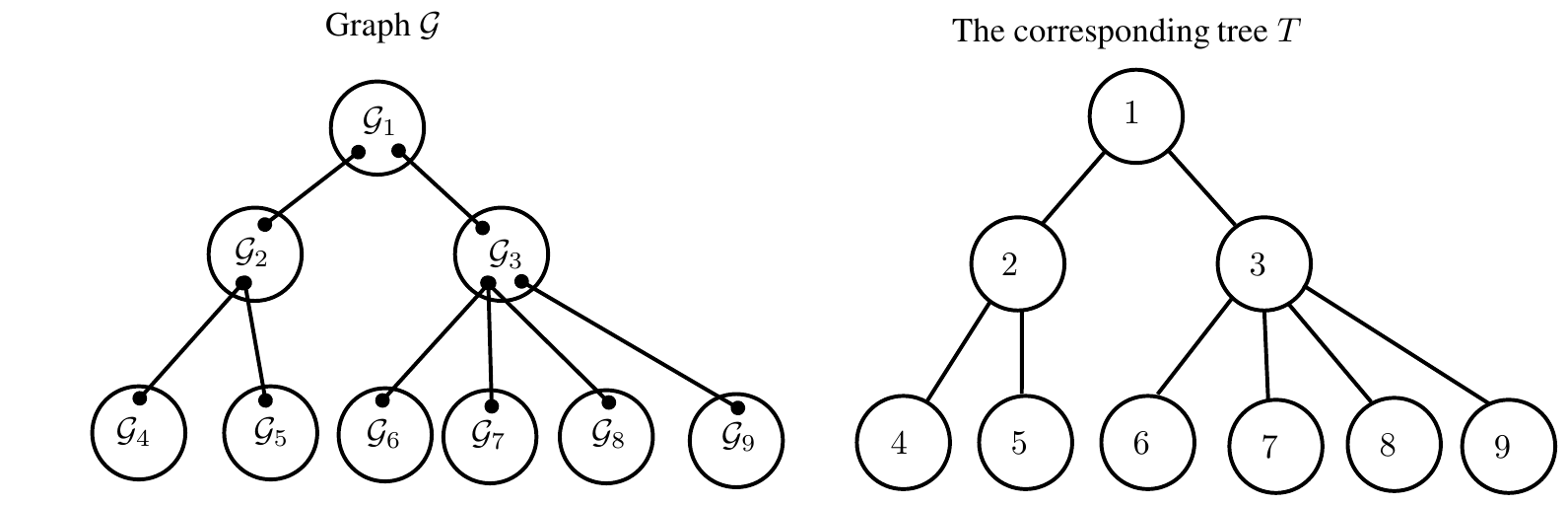}
\caption{A ($\PP$) simple tree structure of a graph $\G$}
\label{fig:tree_structure}
\end{figure}
\end{example}

If a ($\PP$) simple tree structure $\TT = (\Ga, T)$ of $\G$ is given, where $\Ga = [\V_1, \V_2, \ldots, \V_k]$, 
then we can define the following terms:

\begin{enumerate}
	\item Each $\V_i$-induced subgraph $\G_i$ of $\G$ is called a node of $\TT$;\index{node}
	\item If $i$ is the parent of $j$ in $T$, then $\G_i$ is called the parent (node) of $\G_j$
	in $\TT$; We also refer to $\G_j$ as a child (node) of $\G_i$; A node in $\TT$ 
	with no children is called a leaf; The node with no parent is called the root of $\TT$; 
	\item If $j$ is a descendant of $i$ in $T$, then $\G_j$ is called a descendant (node)
	of $\G_i$ and $\G_i$ is called an ancestor (node) of $\G_j$ in $\TT$; 
	\item For each $i \in [k]$ let $\sS_i$ be the subgraph of $\G$ induced by  
	$\V_i \cup (\cup_{j \in \des_T(i)} \V_j)$, where $\des_T(i)$ 
	denotes the set of descendants of $i$ in $T$; 
	In other words, $\sS_i$ is obtained by merging $\G_i$ and all of its
	descendants in $\TT$;
	\item If $\G_j$ is a child of $\G_i$, and $\{u,v\} \in \EG$, where $u \in \V_i$ and $v \in \V_j$,
	then $u$ is called a \emph{downward connector} (DC) of $\G_i$ and $v$ is called the 
	\emph{upward connector} (UC) of $\G_j$; Each node may have several DCs but at most one UC;
	We refer to the DCs and UC of a node as \emph{connectors} of that node. 
	\item Let $\mdct$ denote the maximum number of DCs of a node of $\TT$. 
\end{enumerate}
For instance, for the ($\PP$) simple tree structure depicted in Figure~\ref{fig:tree_structure}, 
suppose that $\G_1$ is the root node, then the node $\G_3$ has two DCs and 
four children. 

Let $\PP$ be a collection of finitely many families of graphs that satisfy (P1), (P2), and (P3).
For any $c > 0$ we define the following family of connected graphs 
\[
\FF_\PP(c) \define \Big\{ \G:\ \G \text{ is connected and has a ($\PP$) simple tree structure } 
\TT \text{ with } \mdct \leq c \Big\}.
\]
A ($\PP$) simple tree structure of a graph $\G$ that proves the membership of $\G$ in 
$\FF_\PP(c)$ is called a \emph{relevant tree structure} of $\G$.
The graph $\G$ in Example~\ref{ex:tree-structure} belongs to $\FF_\PP(2)$. 

\begin{remark}
Suppose that $\PP$ consists of the perfect graphs and the 
outerplanar graphs. Take $\G \in \FF_\PP(c)$ ($c \geq 1$) with a relevant tree structure $\TT$
satisfying the following. There exist a node $\G_i$ of $\TT$ that is perfect
but not outerplanar, and another node $\G_j$ that is outerplanar but not perfect. 
Consequently, $\G$ is neither perfect nor outerplanar. Hence $\G \notin \PP$.
The same argument shows that in general, if $\PP$ contains at least two 
(irredundant) families of graphs then $\FF_\PP(c)$ \emph{properly} contains the families 
of (connected) graphs in $\PP$. Here, a family of graph in $\PP$ is irredundant if
it is not contained in the union of the other families. 
Hence, $\FF_\PP(c)$ ($c \geq 1$) always contains new graphs other than those in $\PP$. 
\end{remark}

\begin{remark}
In general, we can consider \emph{$k$-multiplicity tree structure} of a graph
for every integer $k \geq 0$. 
In such a tree structure, a (parent) node is connected to each of its
child by at most $k$ edges that share the same endpoint in the parent node. 
The $0$-multiplicity tree structures are trivial (see Lemma~\ref{lem:connected_components}). 
The $1$-multiplicity tree structures are simple tree structures. 
In the scope of this paper, we only focus on graphs with simple
tree structures. 
\end{remark}

\subsection{A Polynomial Time Algorithm for Min-Ranks of Graphs in $\FF_\PP(c)$}
\label{subsec:algorithm_for_min-rank}
In this section we show that the min-rank of a member 
of $\FF_\PP(c)$ can be found in polynomial time. 

\vskip 10pt 
\begin{theorem} 
\label{thm:dynamic1}
Let $\PP$ be a collection of finitely many families of graphs that satisfy (P1), (P2), and (P3)
(see Section~\ref{subsec:tree_structure}).
Let $c > 0 $ be a constant and $\G \in \FF_\PP(c)$. Suppose further that a ($\PP$) simple tree structure $\TT = (\Ga, T)$ of $\G$ with $\mdct \leq c$ is known. Then there is an algorithm that computes the min-rank of $\G$ in polynomial time.  
\end{theorem} 

To prove Theorem~\ref{thm:dynamic1}, we describe below an algorithm that computes the min-rank of $\G$ when $\G \in \FF_\PP(c)$ and investigate its complexity. 

First, we introduce some notation which is used throughout this section. If $v$ is any vertex of a graph $\G$, then $\G-v$ denotes the graph obtained from $\G$ by removing $v$ and all edges incident to $v$. In general, if $V$ is any set of vertices, then $\G-V$ denotes the graph obtained from $\G$ by removing all vertices in $V$ and all edges incident to any vertex in $V$. 
In other words, $\G-V$ is the subgraph of $\G$ induced by $\VG \setminus V$. 
Note that if $\G \in \PP$ then the min-rank of $\G-V$ can be computed in 
polynomial time for every subset $V \subseteq \VG$. 
The union of two or more graphs is a graph whose vertex set and edge set are the unions of the vertex sets and of the edge sets of the original graphs, respectively.

The following results from \cite{BerlinerLangberg2011} 
are particularly useful in our discussion.
Their proofs can be found in \cite{BerlinerLangberg-full}, which is the full version of \cite{BerlinerLangberg2011}. 

\vskip 10pt 
\begin{lemma}[\cite{BerlinerLangberg2011}]
\label{lem:Property1} 
Let $v$ be a vertex of a graph $\G$. Then 
\[
\mrq(\G)-1 \leq \mrq(\G-v) \leq \mrq(\G). 
\]
\end{lemma} 
\vskip 10pt  

\begin{lemma}[\cite{BerlinerLangberg2011}]
\label{lem:Property3} 
Let $\G_1$ and $\G_2$ be two graphs with one common vertex $v$. Then 
\[
\begin{split} 
\mrq(\G_1 \cup \G_2) = &\ \mrq(\G_1 - v) + \mrq(\G_2 -v)\\
&+ \big(\mrq(\G_1) - \mrq(\G_1 - v)\big)\times \big(\mrq(\G_2) - \mrq(\G_2 - v)\big). 
\end{split} 
\]
In other words, the min-rank of $\G_1 \cup \G_2$ can be computed explicitly based on the min-ranks of $\G_1$, $\G_1-v$, $\G_2$, and $\G_2 - v$.  
\end{lemma}  
\vskip 15pt 

\nin{\bf Algorithm 1:}\\
Suppose $\G \in \FF_\PP(c)$ and a relevant tree structure $\TT = (\Ga, T)$ of $\G$ is given. 
The algorithm computes the min-rank by dynamic programming in a bottom-up 
manner, from the leaves of $\TT$ to its root. Suppose that $\Ga = 
[\V_1, \V_2, \ldots, \V_k]$ and $\G_i$ is induced by $\V_i$ for $i \in [k]$. 
Let $v_i$ be the UC (if any) of $\G_i$ for $i \in [k]$. 
Recall that $\sS_i$ is the induced subgraph of $\G$ obtained by merging
$\G_i$ and all of its descendants in $\TT$. 
For each $i$, Algorithm 1 maintains a table which contains the 
two values, namely, min-ranks of $\sS_i$ and $\sS_i - v_i$.
The min-rank of the latter is omitted if $\G_i$ is the root node of $\TT$. 
An essential point is that the min-ranks of $\sS_i$ and $\sS_i - v_i$
can be computed in polynomial time from the min-ranks of $\sS_j$'s and 
$(\sS_j - v_j)$'s where $\G_j$'s are children of $\G_i$, and from the min-ranks of
at most $2^c$ induced subgraphs of $\G_i$. 
Each of these subgraphs is obtained from $\G_i$ by removing
a subset of a set that consists of at most $c$ vertices of $\G$.
When the min-rank of $\sS_{i_0}$ is determined, where $G_{i_0}$ 
is the root of $\TT$, the min-rank of $\G = \sS_{i_0}$ is found.\\ 

\nin {\bf At the leaf-nodes:}\\
Suppose $\G_i$ is a leaf and $v_i$ is its UC. 
Since $\G_i$ has no children, $\sS_i \equiv \G_i$. 
Hence, 
\[
\mrq(\sS_i) = \mrq(\G_i),
\]
and 
\[
\mrq(\sS_i-v_i) = \mrq(\G_i - v_i).
\] 
Since $\G_i \in \PP$, the graph $\G_i - v_i$, which is an induced subgraph of $\G_i$, also belongs to $\PP$ 
(according to the property (P1) of $\PP$). 
Therefore, both $\mrq(\G_i)$ and $\mrq(\G_i-v_i)$ can be computed in polynomial time.\\ 

\nin {\bf At the intermediate nodes:}\\
Suppose the min-ranks of $\sS_j$ and $\sS_j-v_j$ are known for every child $\G_j$ of $\G_i$. 
The goal of the algorithm at this step is to compute the min-ranks 
of $\sS_i$ and $\sS_i - v_i$ in polynomial time. 
It is complicated to analyze directly the general case where $\G_i$ has an arbitrary number (at most $c$) of downward connectors. 
Therefore, we first consider a special case where $\G_i$ has only one downward connector (Case~1). The results established in this case are then used to investigate the general case (Case~2). 
\\

\noindent {\bf Case 1:}	$\G_i$ has only one DC $u$ and has 
$r$ children, namely $\G_{j_1}, \G_{j_2}, \ldots, \G_{j_r}$, 
all of which are connected to $\G_i$ via $u$ (Figure~\ref{fig:one_dc}).
\begin{figure}[H]
\centering
\includegraphics{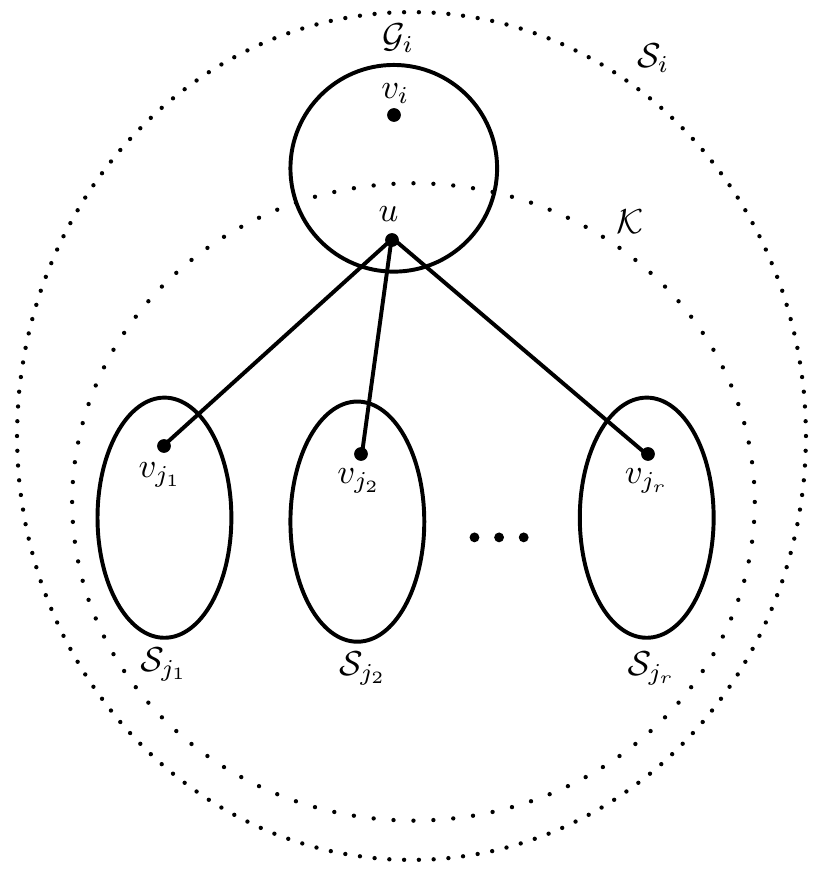}
\caption{$\G_i$ has only one downward connector}
\label{fig:one_dc}
\end{figure}	

Let $\K$ be the subgraph of $\G$ induced by the following set
of vertices
\[
\V(\K) = \V(\sS_{j_1}) \cup \V(\sS_{j_2}) \cup \cdots \cup \V(\sS_{j_r}) \cup \{u\}. 
\] 	
Notice that the graphs $\G_i$ and $\K$ have exactly one vertex in common, namely, $u$. Hence by Lemma~\ref{lem:Property3},
once the min-ranks of $\G_i$, $\G_i - u$, $\K$, and $\K-u$ are known, the min-rank of $\sS_i = \G_i \cup \K$ can be explicitly computed. Similarly, if $v_i \neq u$ and the min-ranks of $\G_i - v_i$, $\G_i - v_i - u$, $\K$, and $\K-u$ are known, the min-rank of $\sS_i - v_i = (\G_i - v_i) \cup \K$ can be explicitly computed. Observe also that 
if $v_i \equiv u$ then by Lemma~\ref{lem:connected_components},
\[
\mrq(\sS_i - v_i ) = \mrq(\G_i-u) + \mrq(\K-u). 
\]
Again by Lemma~\ref{lem:connected_components}, 
\[
\mrq(\K-u) = \sum_{\ell = 1}^r \mrq(\sS_{j_\ell}),
\]	
which is known. Moreover, as $\G_i \in \PP$, the min-ranks of $\G_i$, $\G_i-v_i$, $\G_i - u$, and $\G_i-v_i-u$ 
can be determined in polynomial time. 
Therefore it remains to compute the min-rank of $\K$ efficiently. 
According to the following claim, the min-rank of $\K$ can be explicitly computed based on the knowledge of the min-ranks of $\sS_{j_\ell}$ and $\sS_{j_\ell}- v_{j_\ell}$ for $\ell \in [r]$.
Note that by Lemma~\ref{lem:Property1}, either $\mrq(\sS_{j_\ell}- v_{j_\ell}) = 
\mrq(\sS_{j_\ell})$ or $\mrq(\sS_{j_\ell}- v_{j_\ell}) = \mrq(\sS_{j_\ell}) - 1$, $\ell \in [r]$. 

\vskip 10pt 
\begin{lemma}
\label{lem:claim1}
The min-rank of $\K$ is equal to
\[ 
\begin{cases}
\mrq(\K-u), &\text{ if } \exists h \in [r] \text{ s.t. } \mrq(\sS_{j_h}-v_{j_h}) = \mrq(\sS_{j_h}) - 1,\\
\mrq(\K-u) + 1, &\text{ otherwise.}
\end{cases}
\]
\end{lemma}
\begin{proof} 
Suppose there exists $h \in [r]$ such that 
\[
\mrq(\sS_{j_h}-v_{j_h}) = \mrq(\sS_{j_h}) -1.
\] 
By Lemma~\ref{lem:Property1}, 
\[
\mrq(\K) \geq \mrq(\K-u).
\]
Therefore, in this case
it suffices to show that a matrix that fits $\K$ and has rank equal to $\mrq(\K-u)$ exists. 
Indeed, such a matrix $\bM$ can be constructed as follows. 
The rows and columns of $\bM$ are labeled by the elements in $\V(\K)$ 
(see Definition~\ref{def:mr-def}). 
Moreover, $\bM$ satisfies the following 
properties:
\begin{enumerate}
\item
Its sub-matrix restricted to the rows and columns labeled by the elements in 
$\V(\sS_{j_\ell})$ ($\ell \neq h$) fits $\sS_{j_\ell}$ and 
has rank equal to $\mrq(\sS_{j_\ell})$; 
\item
Its sub-matrix restricted to the rows and columns labeled by the elements in 
$\V(\sS_{j_h}) \setminus \{v_{j_h}\}$ fits $\sS_{j_h}-v_{j_h}$ 
and has rank equal to $\mrq(\sS_{j_h}-v_{j_h})$;
\item
$\bM_{u} = \bM_{v_{j_h}} = \be_u + \be_{v_{j_h}}$, 
where $\be_v$ for $v \in \V(\K)$ denotes the unit vector (with coordinates labeled by
the elements in $\V(\K)$) that has a one at the $v$th coordinate and zeros elsewhere;
Recall that $\bM_u$ denotes the row of $\bM$ labeled by $u$;
\item 
All other entries are zero.  
\end{enumerate}
Since the sets $\V(\sS_{j_\ell})$ ($\ell \neq h$), $\V(\sS_{j_h}) \setminus \{v_{j_h}\}$, 
and $\{u, v_{j_h}\}$ are pairwise disjoint, the above requirements can be met 
without any contradiction arising. Clearly $\bM$ fits $\K$. Moreover, 
\[
\begin{split} 
\rank(\bM) 
&= \sum_{\ell \neq h} \rank\big(\bM_{\V(\sS_{j_\ell})}\big) 
+ \rank\big(\bM_{\V(\sS_{j_h}) \setminus \{v_{j_h}\}}\big)
+\rank\big( \bM_{\{u, v_{j_h}\}} \big)\\
&= \sum_{\ell \neq h}\mrq \big(\sS_{j_\ell}\big) + \mrq\big(\sS_{j_h}-v_{j_h}\big) + 1\\
&= \sum_{\ell \neq h}\mrq\big(\sS_{j_\ell}\big) + \mrq\big(\sS_{j_h}\big)\\
&= \mrq(\K - u). 
\end{split} 
\] 

We now suppose that $\mrq(\sS_{j_\ell} - v_{j_\ell}) = \mrq(\sS_{j_\ell})$ for all $\ell \in [r]$. We prove that 
\[
\mrq(\K) = \mrq(\K - u) + 1
\] 
by induction on $r$.
\begin{enumerate}
\item 
The base case: $r = 1$ (Figure~\ref{fig:base}). In this case, $\G_i$ has only $r =1$ child.  
\begin{figure}[H]
\centering
\includegraphics{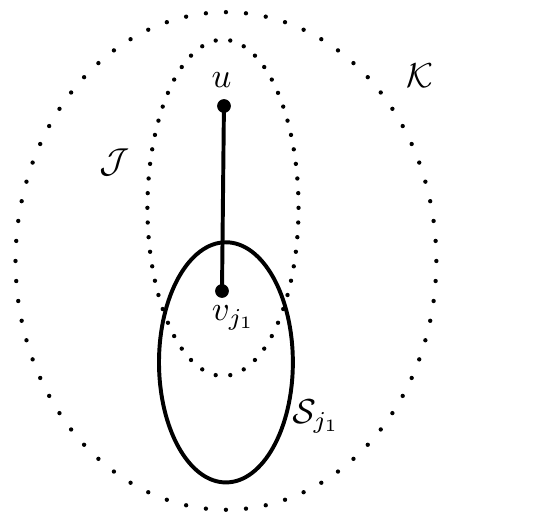}
\caption{The base case when $r = 1$}
\label{fig:base}
\end{figure}
\vskip 10pt 
Let $\J = (\V(\J), \E(\J))$ where $\V(\J) = \{u, v_{j_1}\}$ and 
$\E(\J) = \big\{\{u, v_{j_1}\}\big\}$. Then $\K = \sS_{j_1} \cup \J$ and
$\V(\sS_{j_1}) \cap \V(\J) = \{v_{j_1}\}$. 
Moreover, 
\[
\mrq\big(\sS_{j_1}\big) = \mrq\big(\sS_{j_1}-v_{j_1}\big).
\] 
Therefore by Lemma~\ref{lem:Property3}, 
\[
\begin{split} 
\mrq(\K) &=  \mrq(\sS_{j_1} - v_{j_1}) + \mrq(\J-v_{j_1})\\
&= \mrq(\sS_{j_1}) + 1\\
&= \mrq(\K-u) + 1.
\end{split}
\]
\item 
The inductive step: suppose that the assertion holds for $r \geq 1$. 
We aim to show that it also holds for $r+1$ (Figure~\ref{fig:inductive}). 
\begin{figure}[htb]
\centering
\includegraphics{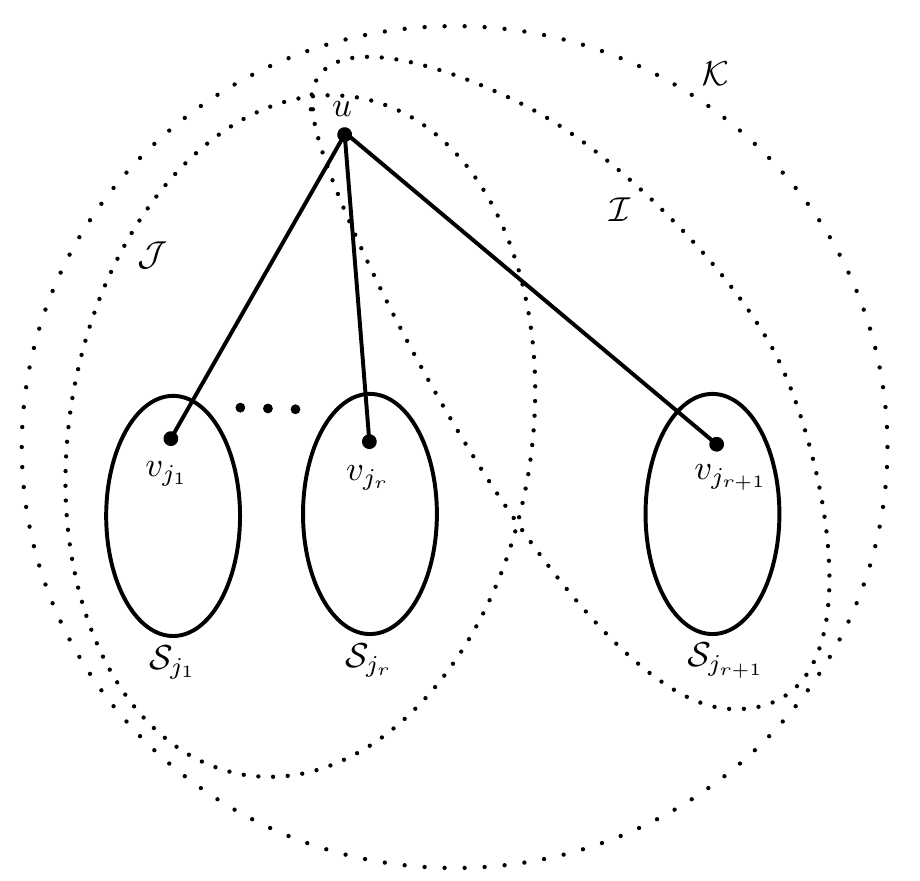}
\caption{The inductive step}
\label{fig:inductive}
\end{figure}

Let $\J$ be the subgraph of $\G$ induced by 
\[
\{u\}\cup \Big(\cup_{\ell =1}^r \V(\sS_{j_\ell}) \Big). 
\]
Since $\mrq(\sS_{j_\ell}-v_{j_\ell}) = \mrq(\sS_{j_\ell})$ for 
all $\ell \in [r]$, by the induction hypothesis, we have
\[
\mrq(\J) = \mrq(\J - u) + 1. 
\]
Let $\I$ be the subgraph of $\G$ induced by $\{u\} \cup \V(\sS_{j_{r+1}})$. As 
\[
\mrq\big(\sS_{j_{r+1}}-v_{j_{r+1}}\big) = \mrq\big(\sS_{j_{r+1}}\big), 
\]
similar arguments as in the base case yield 
\[
\mrq(\I) = \mrq(\I - u) +1. 
\]
Applying Lemma~\ref{lem:Property3} to the graphs $\I$ and 
$\J$ we obtain
\[
\begin{split} 
\mrq(\K) &= \mrq(\I \cup \J)\\
&= \mrq(\I - u) + \mrq(\J-u) + 1\\
&=\mrq\big( \sS_{j_{r+1}} \big) + \sum_{\ell = 1}^r \mrq\big(\sS_{j_\ell}\big) + 1\\
&= \sum_{\ell = 1}^{r+1} \mrq\big(\sS_{j_\ell}\big) + 1, 
\end{split} 
\]
which is equal to $\mrq(\K-u) + 1$.  \qedhere
\end{enumerate}
\end{proof} 
According to the discussion preceding Lemma~\ref{lem:claim1}, Case 1 is settled. 

\vskip 10pt 
\noindent {\bf Case 2:} $\G_i$ has $d$ DCs ($2 \leq d \leq c$), namely, $u_1, u_2, \ldots, u_d$
(Figure~\ref{fig:case2}). 
Let $\{\G_j: \ j \in I_t\}$ be the set of children of $\G_i$ connected to $\G_i$ via $u_t$,
 for $1 \leq t \leq d$.

\begin{figure}[h]
\centering
\includegraphics{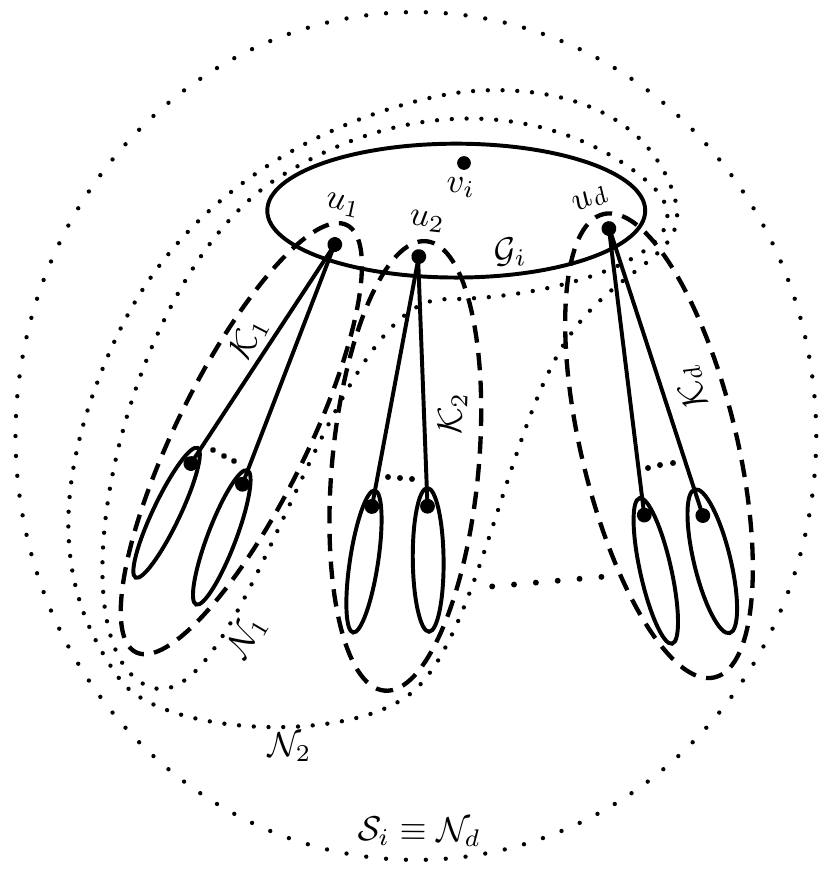}
\caption{$\G_i$ has $d$ downward connectors.  
The solid ellipse on the top represents $\G_i$. The dashed ellipses represent $\K_t$'s.  
The dotted closed curves represent $\N_t$'s. 
}
\label{fig:case2}
\end{figure}

Recall that the goal of the algorithm is to compute the min-ranks of $\sS_i$ and $\sS_i - v_i$ in polynomial time, given that the min-ranks of $\sS_j$ and $\sS_j-v_j$ are known for all children $\G_j$'s of $\G_i$. 

For each $t \in [d]$ let $\K_t$ be the subgraph of $\G$ induced by the following set of vertices
\[
\{u_t\} \cup \Big(\cup_{j\in I_t} \V(\sS_j)\Big). 
\]
As proved in Case 1, based on the min-ranks of $\sS_j$ and $\sS_j-v_j$ for $j \in I_t$, it is 
possible to compute the min-ranks of $\K_t$ and $\K_t - u_t$ explicitly for all $t \in [d]$. 
Let 
\[
\N_1 =  \G_i \cup \K_1,
\] 
and 
\[
\N_t = \N_{t-1} \cup \K_t,
\] 
for every $t \in [d]$ and $t \geq 2$. 
Observe that $\N_d \equiv \sS_i$.
Below we show how the algorithm computes the min-ranks of $\N_d$ and $\N_d-v_i$ recursively in polynomial time.

\vskip 10pt 
\begin{lemma}
\label{lem:recursive}
For every $t \in [d]$ and every $U \subseteq \{v_i, u_{t+1}, u_{t+2}, \ldots, u_d\}$, 
the min-rank of $\N_t - U$ can be calculated in polynomial time.   
\end{lemma} 
\begin{proof} 
\mbox{}
\begin{enumerate}
\item
At the base case, the min-ranks of $\N_1 - U$, for every subset $U \subseteq \{v_i, u_2, u_3,$ $\ldots, u_d\}$, are computed as follows. 

If $v_i \equiv u_1$ and $v_i \in U$, then 
\[
\N_1 - U =  (\G_i - U) \cup (\K_1 - u_1). 
\]
Since 
\[
\V(\G_i - U) \cap \V(\K_1-u_1) = \varnothing, 
\]
by Lemma~\ref{lem:connected_components}, 
\[
\mrq(\N_1-U) = \mrq(\G_i - U) + \mrq(\K_1 - u_1),  
\]
which is computable in polynomial time. 

Suppose that either $v_i \not\equiv u_1$ or $v_i \notin U$. By Lemma~\ref{lem:Property3}, since 
\[
\N_1 - U = (\G_i - U) \cup \K_1,
\]
and 
\[
\V(\G_i - U) \cap \V(\K_1) = \{u_1\},
\]
the min-rank of $\N_1-U$ can be determined based on the min-ranks of 
$\G_i-U$, $\G_i-U-u_1$, $\K_1$, and $\K_1-u_1$. 
The min-ranks of these graphs are either known or computable in polynomial time.
As $\mdct \leq c$, there are at most $2^d \leq 2^c$ (a constant) such subsets $U$. 
Hence, the total computation in the base case can be done in polynomial time.
 
\item
At the recursive step, suppose that the min-rank of $\N_{t-1}-U$, $ t \geq 2$, 
for every subset $U \subseteq \{v_i, u_t, u_{t+1},\ldots, u_d\}$ is known. 
Our goal is to show that the min-rank of $\N_t-V$ for every subset 
$V \subseteq \{v_i, u_{t+1}, u_{t+2}, \ldots, u_d\}$ can be determined in polynomial time. 
Note that there are at most $2^d \leq 2^c$ such subsets $V$. 

If $v_i \equiv u_t$ and $v_i \in V$, then
\[
\N_t - V = (\N_{t-1} - V) \cup (\K_t-u_t). 
\]
Moreover, as we have
\[
\V(\N_{t-1} - V) \cap \V(\K_t-u_t) = \varnothing,
\]
by Lemma~\ref{lem:connected_components}, 
\[
\mrq(\N_t - V) = \mrq(\N_{t-1} - V) + \mrq(\K_t-u_t),
\] 
which is known. Note that $\mrq(\N_{t-1} - V)$ is known from the previous recursive step
since 
\[
V \subseteq \{v_i, u_{t+1}, u_{t+2}, \ldots, u_d\} \subseteq \{v_i, u_t, u_{t+1},\ldots, u_d\}.
\]

Suppose that either $v_i \not\equiv u_t$ or $v_i \notin V$. 
Since 
\[
\N_t - V= (\N_{t-1} - V) \cup \K_t, 
\]
and 
\[
\V(\N_{t-1}-V) \cap \V(\K_t) = \{u_t\}, 
\]
the min-rank of $\N_t-V$ can be computed based on the min-ranks of 
$\N_{t-1}-V$, $\N_{t-1}-V-u_t$, $\K_t$, and $\K_t - u_t$, which are all available from the previous recursive step.   
\qedhere
\end{enumerate}
\end{proof} 
\vskip 10pt 

When the recursive process described in Lemma~\ref{lem:recursive} reaches $t = d$, 
the min-ranks of $\N_d$ and $\N_d- v_i$ are found, as desired. 
Moreover, as there are $d \leq c$ steps, and in each step, the computation can be done
in polynomial time, we conclude that the min-ranks of these graphs can be found
in polynomial time. 
The analysis of Case 2 is completed. 	

Let $\PP$ be a collection of finitely many families of graphs that satisfy (P1), (P2), and (P3)
(see Section~\ref{subsec:tree_structure}).
For any $c > 0$, let $\FF_\PP(c\log(\cdot))$ denote the following family of graphs 
\[
\Big\{ \G:\ \G \text{ is connected and has a ($\PP$) simple tree structure } 
\TT \text{ with } \mdct \leq c \log |\VG| \Big\}.
\]
Note that $\FF_\PP(c\log(\cdot))$ properly contains $\FF_\PP(c)$ as a sub-family. 
If $\G \in \FF_\PP(c\log(\cdot))$
for some constant $c > 0$, then the time complexity of Algorithm~1 is still polynomial
in $n = |\VG|$. 
Indeed, since $2^d \leq 2^{c\log n} = n^c$, Lemma~\ref{lem:recursive} still holds. 
As all other tasks in Algorithm~1 require polynomial time in $n$, we conclude that
the running time of the algorithm is still polynomial in $n$. 
However, as discussed in Section~\ref{subsec:recognition_algo}, 
we are not able to find a polynomial time algorithm to recognize a 
graph in $\FF_\PP(c\log(\cdot))$. 

\vskip 10pt
\begin{theorem}
\label{thm:dynamic2}
Let $\PP$ be a collection of finitely many families of graphs that satisfy (P1), (P2), and (P3)
(see Section~\ref{subsec:tree_structure}).
Let $c > 0 $ be a constant and $\G \in \FF_\PP(c\log(\cdot))$. Suppose further that a ($\PP$) simple tree structure $\TT = (\Ga, T)$ of $\G$ with $\mdct \leq c\log |\VG|$ is known. Then there is an algorithm that computes the min-rank of $\G$ in polynomial time.  
\end{theorem}
\vskip 10pt

\subsection{An Algorithm to Recognize a Graph in $\FF_\PP(c)$}
\label{subsec:recognition_algo}

In order for Algorithm~1 to work, it is assumed that a relevant tree structure of the input graph $\G \in \FF_\PP(c)$ is given. 
Therefore, the next question is how to design an algorithm that recognizes a graph in that family and
subsequently finds a relevant tree structure for that graph in polynomial time. 

\vskip 10pt 
\begin{theorem}
\label{thm:recognition1}
Let $\PP$ be a collection of finitely many families of graphs that satisfy (P1), (P2), and (P3)
(see Section~\ref{subsec:tree_structure}).
Let $c > 0$ be any constant. Then there is a polynomial time algorithm 
that recognizes a member of $\FF_\PP(c)$. Moreover, this algorithm also
outputs a relevant tree structure of that member.   
\end{theorem} 
\vskip 10pt 
In order to prove Theorem~\ref{thm:recognition1}, we introduce Algorithm 2 (Figure~\ref{fig:Algo-2}). 
This algorithm consists of two phases: Splitting Phase (Figure~\ref{fig:Algo-2-Splitting}), 
and Merging Phase (Figure~\ref{fig:Algo-2-Merging}).
\begin{figure}[h]
\centering{
\fbox{
\parbox{5in}{
\nin{\bf Algorithm 2:}\\
\nin{\bf Input:} A connected graph $\G=(\VG, \EG)$ and a constant $c > 0$.\\
\nin{\bf Output:} If $\G \in \FF_\PP(c)$, the algorithm prints out a confirmation message, namely ``$\G \in \FF_\PP(c)$'', and then returns a relevant tree structure of $\G$. 
Otherwise, it prints out an error message ``$\G \notin \FF_\PP(c)$''. \\
\nin{\bf Splitting Phase}\\
\nin{\bf Merging Phase}\\
}
}
}
\caption{Algorithm 2}
\label{fig:Algo-2}
\end{figure}
The general idea behind Algorithm 2 is the following. 
Suppose $\G \in \FF_\PP(c)$ and $\TT$ is a relevant tree structure of $\G$. 
In the Splitting Phase, the algorithm \emph{splits} $\G$ into a number of components
(induced subgraphs), which form the set of nodes of a ($\PP$) simple tree structure $\TT'$
of $\G$. It is possible that $\mdc(\TT') > c$, that is, $\TT'$ is not a relevant tree structure
of $\G$. However, it can be shown that each node of $\TT'$ is actually an induced
subgraph of some node of $\TT$. Based on this observation, the main task 
of the algorithm in the Merging Phase is to merge suitable nodes of $\TT'$
in order to turn it into a relevant tree structure of $\G$. Note though that this tree structure
might not be the same as $\TT$. 

\begin{figure}[htb]
\centering{
\fbox{
\parbox{5.5in}{
\nin{\bf Splitting Phase:}\\
\nin{\bf Initialization:} Create two empty queues, $\Q_1$ and $\Q_2$, 
which contains graphs as their elements. Push $\G$ into $\Q_1$.\mbox{}
\begin{algorithmic}
\WHILE {$\Q_1 \neq \varnothing$}
\FOR {$\A = (\V(\A), \E(\A)) \in \Q_1$}
\STATE{Pop $\A$ out of $\Q_1$;}
\IF {there exist $U$ and $\V$ that partition $\V(\A)$ and $\s_\A(U,V)=1^*$}
   \STATE{Let $\B$ and $\C$ be subgraphs of $\A$ induced by $U$ and $V$, respectively;}
	 \STATE{Push $B$ and $C$ into $\Q_1$;}
\ELSIF {$\A \in \PP$}
       \STATE{Push $\A$ into $\Q_2$;}
\ELSE \STATE{Print the error message ``$\G \notin \FF_\PP(c)$'' and exit;}
\ENDIF
\ENDFOR
\ENDWHILE \\
Suppose $\Q_2$ contains $h$ graphs $\A_1, \A_2, \ldots, \A_{h}$. 
Let $T'$ be a graph with $\V(T') = [h]$ and 
$\E(T') = \big\{\{\ell, m\}:\ \s_\G(\V(\A_\ell), \V(\A_m)) = 1\big\}$. 
\end{algorithmic}
}
}
}
\caption{Algorithm 2 -- Splitting Phase}
\label{fig:Algo-2-Splitting}
\end{figure}
\begin{figure}[htb]
\centering
\includegraphics{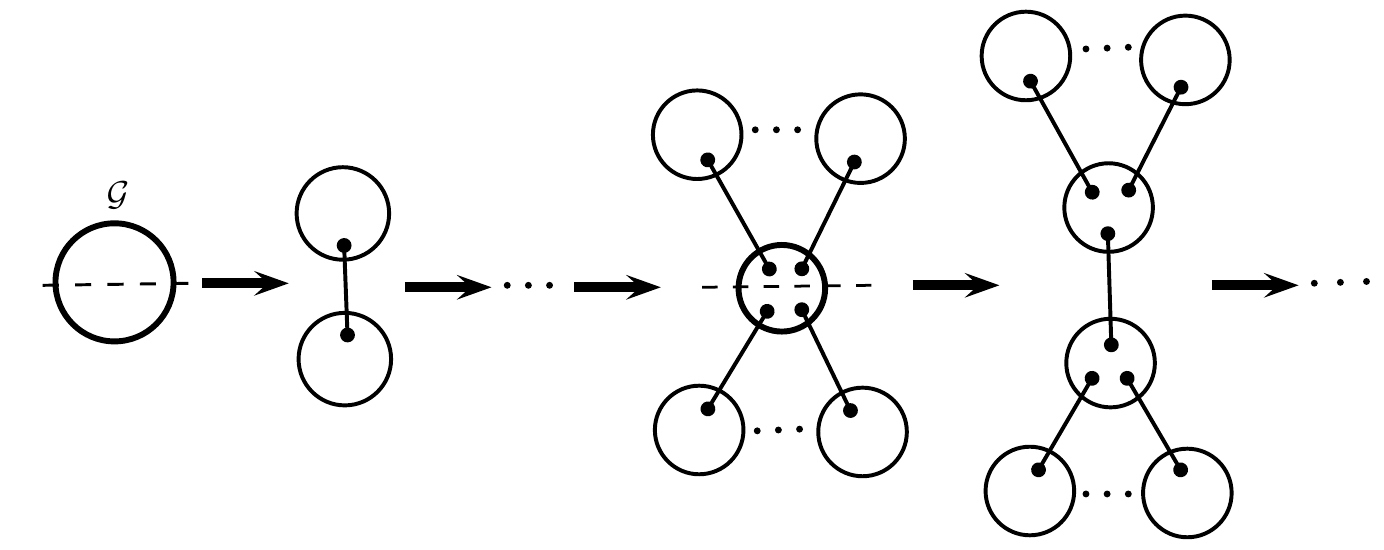}
\caption{Splitting Phase of Algorithm 2}
\label{fig:Splitting_Phase}
\end{figure}
\let\thefootnote\relax\footnotetext{$^*$ This condition is equivalent to that of $\A$ having a bridge}
Suppose $\G$ successfully passes the Splitting Phase, that is, 
no error messages are printed out during this phase. 
In the Splitting Phase, the algorithm first splits $\G$ 
into two components (induced subgraphs) that are connected to each other by exactly
one edge (bridge) in $\G$. 
It then keeps splitting the existing components, whenever possible, 
each into two new smaller components that are connected to each other 
by exactly one edge in the original component (see Figure~\ref{fig:Splitting_Phase}). 
A straightforward inductive argument shows the following:
\begin{enumerate}
	\item Throughout the Splitting Phase, the vertex sets that induce 
	        the components of $\G$ partition $\VG$; Hence $\V(\A_m)$'s, $m \in [h]$, partition $\VG$; 
	\item Throughout the Splitting Phase, any two different components of $\G$ 
	are connected to each other by at most one edge in $\G$; 
	Therefore, $\s_\G(\V(\A_\ell), \V(\A_m)) \in \{0,1\}$ for every $\ell \neq m$, $\ell, m \in [h]$;
	\item At any time during the Splitting Phase, the graph that is obtained from $\G$ 
	by contracting the vertex set of each component of $\G$ to a single vertex is a tree; 
	Therefore, $T'$ is a tree;
  \item	Throughout the Splitting Phase, every component of $\G$ remains connected; 
\end{enumerate}
It is also clear that each $\A_m$ ($m \in [h]$) belongs to a family in $\PP$. 
Since $\G$ passes the Splitting Phase successfully, $\TT' = (\Ga' = [\V(\A_1), \ldots, \V(\A_{h})], T')$ 
is already qualified to be a ($\PP$) simple tree structure of $\G$. 

\vskip 10pt 
\begin{lemma}  
\label{lem:claim2}
Suppose $\G \in \FF_\PP(c)$ and $\TT = (\Ga = [\V_1, \V_2, \ldots, \V_k], T)$
is a relevant tree structure of $\G$. Then at any time during the Splitting Phase, for any $\A \in \Q_1$, 
either of the following two conditions must hold:
\begin{enumerate}
	\item $\A$ has a bridge;
	\item $\V(\A) \subseteq \V_i$ for some $i \in [k]$. 
\end{enumerate}
\end{lemma}
\begin{proof}
Suppose the second condition does not hold. 
Since $\V_1, \V_2, \ldots, \V_k$ partition $\VG \supseteq \V(\A)$,
there exist some $r \geq 2$ and some subset $\{i_1, i_2, \ldots, i_r\}$ of $[k]$
such that 
\[
\V(\A) \subseteq \cup_{\ell = 1}^r \V_{i_\ell},
\]
and  
\[
\V(\A) \cap \V_{i_\ell} \neq \varnothing, \ \forall \ell \in [r].
\]
We are to show that $\A$ has a bridge. 
Without loss of generality, suppose that $\G_{i_r}$ has no children (in $\TT$) among
$\G_{i_1},\G_{i_2},\ldots, \G_{i_{r-1}}$. Let $U = \V(\A) \cap \V_{i_r} \neq \varnothing$ and 
$V = \V(\A) \cap \cup_{1 \leq \ell \leq r-1}\V_{i_\ell} \neq \varnothing$. 
Then
\[ 
\s_\A(U, V) \leq \s_\G(\V_{i_r}, \cup_{1 \leq \ell \leq r-1}\V_{i_\ell}) \leq 1, 
\]
where the second inequality follows from the property of a ($\PP$) simple tree structure
and from the assumption that $\G_{i_r}$ has no children among
$\G_{i_1},\G_{i_2},\ldots, \G_{i_{r-1}}$.
As $U \cup V = \V(\A)$ and $\A$ is connected, it must hold that $\s_\A(U,V) = 1$. 
Hence, $\A$ has a bridge.  
\end{proof}

\vskip 10pt 
\begin{lemma}
\label{lem:phase1_terminates}
If $\G \in \FF_\PP(c)$ then $\G$ passes the Splitting Phase successfully. 
\end{lemma}
\begin{proof}
Suppose $\TT = (\Ga = [\V_1, \V_2, \ldots, \V_k], T)$
is a relevant tree structure of $\G$.
By Lemma~\ref{lem:claim2}, for any $\A \in \Q_1$, either $\A$ has a bridge or 
$\V(\A) \subseteq \V_i$ for some $i \in [k]$. The latter condition
implies that $\A$ is an induced subgraph of $\G_i$, and hence,   
$\A \in \PP$. Therefore, $\G$ passes the Splitting Phase without 
any error message printed out. 
\end{proof} 

\vskip 10pt 
\begin{lemma} 
\label{lem:claim3}
Suppose $\G \in \FF_\PP(c)$ and $\TT = (\Ga = [\V_1, \V_2, \ldots, \V_k], T)$
is a relevant tree structure of $\G$. Then for each $m \in [h]$, there exists
a unique $i \in [k]$ such that $\V(\A_m) \subseteq \V_i$. 
\end{lemma}
\begin{proof}
According to the algorithm, $\A_m$ does not have any bridge for every $m \in [h]$. 
By Lemma~\ref{lem:claim2}, for each $m \in [h]$, $\V(\A_m) \subseteq \V_i$ for some $i \in [k]$. 
The uniqueness of such $i$ follows from the fact that $\V_i \cap \V_j = \varnothing$
for every $i \neq j$. 
\end{proof}
\vskip 10pt

\begin{figure}[htb]
\centering{
\fbox{
\parbox{5.5in}{
\nin{\bf Merging Phase:}
\begin{algorithmic}
\FOR {$r = 1$ to $h$}
    \STATE{Let $\TT'_r$ be a copy of $\TT'$;}
		\STATE{Assign $\A_r$ to be the root node of $\TT'_r$;}
		\IF {$\mdc(\TT'_r) \leq c$}
		     \STATE{Print ``$\G \in \FF_\PP(c)$'', return $\TT'_r$, and exit;}
		\ELSE
		      \STATE{Let $\cL_r$ be an ordered list of nodes of $\TT'_r$ such that every node appears in the list later
					than all of its children;}
		      \FOR {$\A_m \in \cL_r$}
					     \STATE{Let $D_m$ be the list of all $\A_m$'s DCs;}
							\STATE{Find a maximum subset $E_m$ of $D_m$ with $|E_m| \geq |D_m| - c$, such that}
							    \STATE{
									\begin{itemize}
                      \item[] 1) The set $C_m$ of all children of $\A_m$ connected to $\A_m$ via DCs
									             in $E_m$ consists of only leaf nodes, and
							        \item[] 2) The set $\V(\A_m) \cup \big( \cup_{\A_\ell \in C_m} \V(\A_\ell)\big)$
							                 induces a subgraph of $\G$ which belongs to $\PP$;
							    \end{itemize}}
							\IF {there exists such a set $E_m$}
							      \STATE{Merge $\A_m$ and its children in $C_m$;}
							\ELSIF {$r = h$}
										      \STATE{Print ``$\G \notin \FF_\PP(c)$'' and exit;}
							\ELSE
													\STATE{Return to the outermost ``for'' loop;}
							\ENDIF
					\ENDFOR
				\STATE{Print ``$\G \in \FF_\PP(c)$'', return $\TT'_r$, and exit;}
		\ENDIF
\ENDFOR
\end{algorithmic}
}
}
}
\caption{Algorithm 2 -- Merging Phase}
\label{fig:Algo-2-Merging}
\end{figure}

As discussed earlier, after a successful completion of the Splitting Phase, 
a ($\PP$) simple tree structure of $\G$, that is
$\TT' = (\Ga' = [\V(\A_1), \ldots, \V(\A_{h})], T')$, 
is obtained. 
In the Merging Phase (Figure~\ref{fig:Algo-2-Merging}), the algorithm first assigns a root
node for $\TT'$. It then \emph{traverses} $\TT'$ in a bottom-up 
manner, tries to \emph{merge} every node it visits with a suitable set
of the node's leaf child-nodes (if any) to reduce the number of DCs of the node
below the threshold $c$. 
If such a set of children of the node cannot be found, then the algorithm 
restarts the whole merging process by assigning a different root node to the 
(original) tree structure $\TT'$ and traversing the tree structure again, from the leaves to the root. 
The algorithm stops when 
it finds a relevant tree structure, whose maximum number of DCs of every node is
at most $c$. 
If no relevant tree structures are found after trying out all possible assignments
for the root node, the algorithm claims that $\G \notin \FF_\PP(c)$ and exits. 

To preserve the tree structure $\TT'$ 
throughout the phase, only a copy of it, namely $\TT'_r$, is used when 
the node $\A_r$ is assigned as a root. 
Let $\cL_r$ be an ordered list of nodes of $\TT'_r$ such that every node appears
in the list later than all of its child-nodes.   
The algorithm visits each node in the list sequentially. 
The merging operation is described in more details as follows. 
Suppose $\A_m$ is the currently visited node, and $C_m$ is a set
of its leaf child-nodes, which is to be merged. The merging operation 
enlarges $\A_m$ by merging its vertex set with the vertex sets $\V(\A_\ell)$ for
all $\A_\ell \in C_m$. At the same time, the node $\A_\ell$ is deleted from the tree 
structure $\TT'_r$ for every $\A_\ell \in C_m$.
Observe that since $\A_m$ can only be merged with its leaf child-nodes,
no new DCs are introduced as a result of the merging operation.
Therefore, the merging operation never increases the number of 
DCs of the visited node. 
Observe also that a new ($\PP$) simple tree structure of $\G$
is obtained after every merging operation.

\vskip 10pt 
\begin{lemma}
\label{lem:membership_proof}
If Algorithm 2 terminates successfully then $\G \in \FF_\PP(c)$. 
\end{lemma}
\begin{proof} 
According to the discussion before Lemma~\ref{lem:claim2}, if $\G$ passes the Splitting Phase 
successfully then $\TT'$ is a ($\PP$) simple tree structure of $\G$. 
Suppose $\G$ also passes the Merging Phase successfully. 
According to the algorithm, 
there exists a copy $\TT'_r$ of $\TT'$ with root node $\A_r$
such that either $\mdc(\TT'_r) \leq c$ (hence $\G \in \FF_\PP(c)$)
or the following condition holds. 
At \emph{every} node $\A_m$ of $\TT'_r$ that the algorithm visits during the Merging Phase, 
there always exists a set of DCs $E_m$ of $\A_m$ satisfying:
\begin{enumerate}
 \item The set $C_m$ of all children of $\A_m$ connected to $\A_m$ via DCs
				in $E_m$ consists of only leaf nodes, and
 \item The set $\V(\A_m) \cup \big( \cup_{\A_\ell \in C_m} \V(\A_\ell)\big)$
							                 induces a subgraph of $\G$ which belongs to $\PP$.
\end{enumerate}
Moreover, $|E_m| \geq |D_m| - c$, where $D_m$ is the set of DCs of $\A_m$ in $\TT'_r$. 
Therefore, after merging $\A_m$ and its leaf child-nodes in $C_m$, 
$\A_m$ has $|D_m| - |E_m| \leq c$ DCs.
As this situation applies for every node $\A_m$ of $\TT'_r$, 
once the algorithm reaches the root node $\A_r$,  
we obtain a relevant tree structure of $\G$, which 
proves the membership of $\G$ in $\FF_\PP(c)$. 
\end{proof} 

\vskip 10pt 
\begin{lemma}
\label{lem:algo-2_terminates}
If $\G \in \FF_\PP(c)$ then Algorithm 2 terminates successfully.  
\end{lemma} 
\vskip 10pt 

To prove Lemma~\ref{lem:algo-2_terminates}, we need a few more observations.
We hereafter assume that $\G \in \FF_\PP(c)$ and $\TT = (\Ga = [\V_1, \V_2, \ldots, \V_k], T)$
is a relevant tree structure of $\G$.
By Lemma~\ref{lem:claim3}, for each $m \in [h]$, there exists a unique $i \in [k]$ such that 
$\V(\A_m) \subseteq \V_i$. 
Then $i_{\TT}(\A_m) \define i$ is called the \emph{$\TT$-index} of $\A_m$.
For brevity, we often use $i(m)$ to refer to $i_{\TT}(\A_m)$.   
The $\TT$-index of a node $\A_m$ is simply the index of the node in the tree structure $\TT$
that contains $\A_m$ as an induced subgraph. 
Hence we always have $\V(\A_m) \subseteq \V_{i(m)}$ for every $m \in [h]$. 

From now on, let $r_0 \in [h]$ be such that $i(r_0) = i_0$, which is the root of the tree $T$. 
Moreover, suppose $\A_{r_0}$ is assigned to be the root node in $\TT'$.
Recall that $\G_i$ denotes the $\V_i$-induced subgraph of $\G$ ($i \in [k]$). 

\vskip 10pt 
\begin{lemma}
\label{lem:claim4} 
If $\A_\ell$ is a child of $\A_m$ in $\TT'$ then either 
$i(\ell) = i(m)$ or $\G_{i(\ell)}$ is a child of $\G_{i(m)}$ in $\TT$. 
\end{lemma} 
\begin{proof}
We prove this claim by induction on the node $\A_m$. 

\nin{\bf Base case:} Let $m = r_0$ and let $\A_\ell$ be a child of the root node $\A_m$.
Since $\s_\G(\V(\A_\ell), \V(\A_m)) = 1$ and $\V(\A_m) \subseteq \V_{i(m)}=\V_{i_0}$, 
we conclude that either $\V(\A_\ell) \subseteq \V_{i_0}$ or $\V(\A_\ell) \subseteq \V_i$
for some child-node $\G_i$ of $\G_{i_0}$ (in $\TT$).
Therefore, either $i(\ell) = i_0 = i(m)$ or 
$\G_{i(\ell)}=\G_i$ is a child of $\G_{i(m)} = \G_{i_0}$.

\nin{\bf Inductive step:} Suppose the assertion of Lemma~\ref{lem:claim4} holds for all ancestors $\A_{m'}$
(and their corresponding children $\A_{\ell'}$) of $\A_m$. 
Take $\A_\ell$ to be a child of $\A_m$. We aim to show
that the assertion also holds for $\A_m$ and $\A_\ell$. 

As $\s_\G(\V(\A_\ell), \V(\A_m)) = 1$, there are three cases to consider, due to Lemma~\ref{lem:claim3}. 

\nin{\bf Case 1:} There exists some $i \in [k]$ such that $\V(\A_\ell) \subseteq \V_i$ and 
	    $\V(\A_m) \subseteq \V_i$. Then $i(\ell) = i(m) = i$. 

\nin{\bf Case 2:} There exist $i\in [k]$ and $j \in [k]$ such that $\V(\A_\ell) \subseteq \V_j$, 
       $\V(\A_m) \subseteq \V_i$, and $\G_j$ is a child of $\G_i$. In this case, since 
			$i(\ell) = j$ and $i(m) = i$, we deduce that $\G_{i(\ell)}$ is a child of $\G_{i(m)}$. 

\nin{\bf Case 3:} There exist $i\in [k]$ and $j \in [k]$ such that $\V(\A_\ell) \subseteq \V_i$, 
	    $\V(\A_m) \subseteq \V_j$, and $\G_j$ is a child of $\G_i$. We are to derive a contradiction
			in this case. 
			
			Since $\G_{i(m)}=\G_j$ is a child of $\G_{i(\ell)} = \G_i$, $i(m) \neq i_0$. 
			Thus $\A_m$ has at least one ancestor, namely $\A_{r_0}$ ($i(r_0) = i_0$), with a different $\TT$-index.  
			Let $\A_p$ be the closest ancestor of $\A_m$ that satisfies $i(p) \neq i(m) = j$. 
			Then the child $\A_q$ of $\A_p$ that lies on the path from $\A_p$ down to 
			$\A_m$ must have $\TT$-index $i(q)= i(m) = j$. 
			By the inductive hypothesis, $\G_j = \G_{i(m)} =\G_{i(q)}$ is a child of $\G_{i(p)}$ in $\TT$. 
		 			Therefore, $i(p) \equiv i$. 
			\begin{figure}[H]
			\centering
			\includegraphics{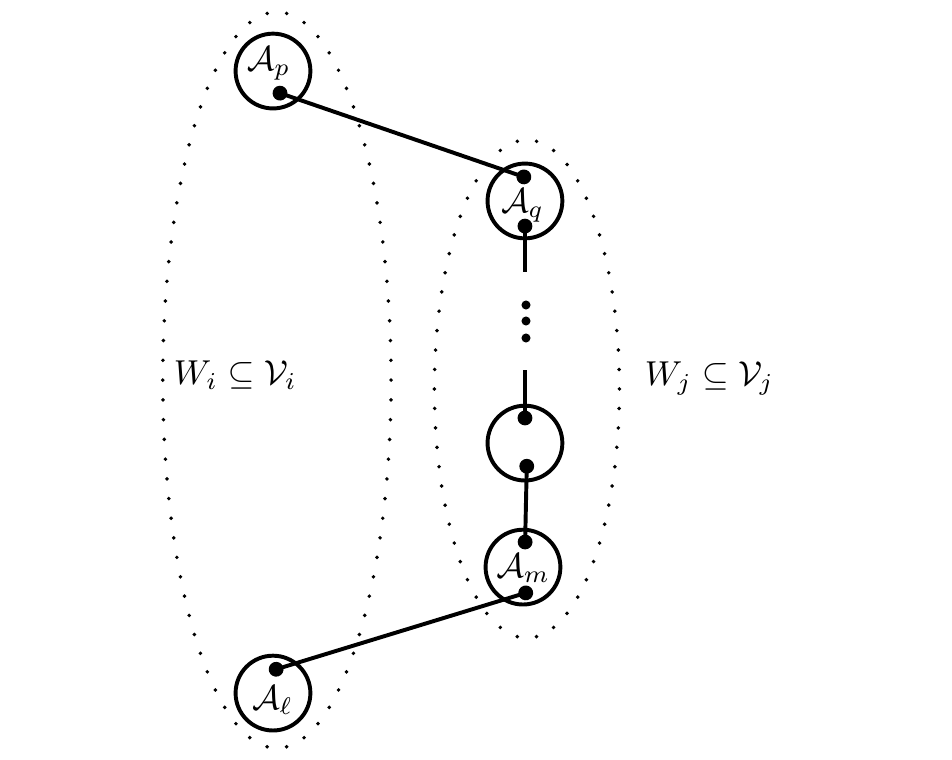}
\caption{Case 3}
\label{fig:case3}
\end{figure}
Let 
			\[
			S = \{s:\ \A_s \text{ is a descendant of } \A_p \text{ and an ancestor of } \A_m\}. 
			\]
			By the definition of $\A_p$, we have $i(s) = j$ for all $s \in S$. Therefore
			\[
			W_j \define \V(A_m) \cup \big( \cup_{s \in S} \V(\A_s) \big) \subseteq \V_j. 
			\]
			Moreover, since $i(p) = i = i(\ell)$, 
			\[
			W_i \define \V(\A_p) \cup \V(\A_\ell) \subseteq \V_i. 
			\]
			Hence 
			\[
			1 = \s_\G(\V_i, \V_j) \geq \s_\G(W_i, W_j) \geq 2, 
			\]
			which is impossible. 
			The last inequality is explained as follows. The two different edges that connect
			$\A_p$ and $\A_q$, $\A_\ell$ and $\A_m$ both have one end in $W_i$
			and the other end in $W_j$ (Figure~\ref{fig:case3}). 
\end{proof} 
\vskip 10pt 

For each $\ell \in [h]$ let $\B_\ell$ be the collection of nodes of $\TT'$ that 
consists of $\A_\ell$ and all of its descendant nodes in $\TT'$. If $\A_\ell$ is a child of $\A_m$, 
	we refer to $\B_\ell$ as a \emph{branch} of $\A_m$ in $\TT'$. 
A branch of $\A_m$ is called \emph{nonessential} if all of its nodes have the same $\TT$-index as $\A_m$. 
Otherwise it is called	\emph{essential}.
A DC of $\A_m$ that connects it to at least one of its essential branches is called an \emph{essential} DC.
Otherwise it is called a \emph{nonessential} DC. 
 
\vskip 10pt 
\begin{lemma}
\label{lem:claim5}
For each $m \in [h]$, the number of essential DCs of 
$\A_m$ is at most $c$. 
\end{lemma} 
\begin{proof}
Suppose by contradiction that some node $\A_m$ has more than $c$ essential DCs. 
Our goal is to show that in $\TT$, the number of DCs of $\G_{i(m)}$ would be larger than $c$, 
which is impossible.
\begin{figure}[H]
\centering
\includegraphics{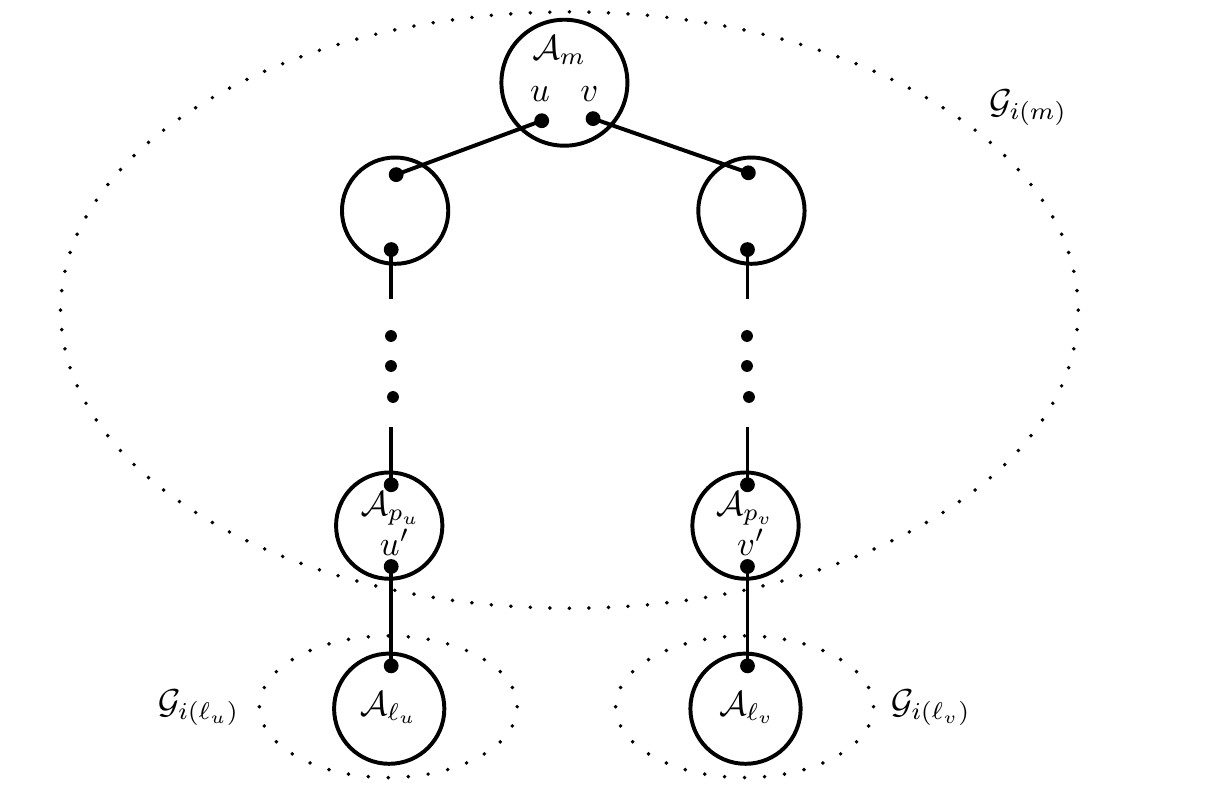}
\caption{DCs of $\A_m$ in $\TT'$ and corresponding DCs of $\G_{i(m)}$ in $\TT$}
\label{fig:claim5-fig}
\end{figure}
For each essential DC $u$ of $\A_m$, let $\A_{\ell_u}$ be the closest
descendant of $\A_m$ (connected to $\A_m$ via $u$) whose 
$\TT$-index is different from that of $\A_m$. In other words, 
$i(\ell_u) \neq i(m)$. Let $\A_{p_u}$ be the parent node of $\A_{\ell_u}$. 
Then clearly $i(p_u) = i(m)$. By Lemma~\ref{lem:claim4}, $\G_{i(\ell_u)}$ is a child of $\G_{i(m)} = \G_{i(p_u)}$ in $\TT$. 
Let $u'$ be the DC that connects $\A_{p_u}$ and $\A_{\ell_u}$ in $\TT'$. 
Note that $u'$ and $u$ are identical when $p_u \equiv m$. 
Since $\V(\A_{p_u}) \subseteq \V_{i(p_u)} = \V_{i(m)}$ and $\V(\A_{\ell_u}) \subseteq \V_{i(\ell_u)}$, 
$u'$ is also the DC that connects $\G_{i(m)}$ and its child $\G_{i(\ell_u)}$ in $\TT$.

We use similar notations for another essential DC $v \neq u$ of $\A_m$. 
Then another child of $\G_{i(m)}$, namely $\G_{i(\ell_v)}$, is connected to 
$\G_{i(m)}$ via the DC $v'$ of $\G_{i(m)}$ (Figure~\ref{fig:claim5-fig}). 

If either $u'$ or $v'$ does not belong to $\A_m$, then as $\TT'$ is a ($\PP$) simple tree structure of $\G$, 
it is straightforward that $u' \neq v'$. 
If both of the DCs are in $\A_m$ then $u' \equiv u$ and $v' \equiv v$, which in turn implies that $u' \neq v'$. 
Hence, distinct essential DCs of $\A_m$ in $\TT'$ correspond to distinct DCs of $\G_{i(m)}$ in $\TT$. 
Therefore, $\G_{i(m)}$ would have more than $c$ DCs in $\TT$.     
\end{proof} 

\vskip 10pt 
\begin{lemma} 
\label{lem:claim6}
In the Merging Phase the algorithm merges each nonessential branch of $\TT'$ 
into a leaf. 
\end{lemma} 
\begin{proof}
Suppose $\B_\ell$ is a nonessential branch of $\TT'$. 
All nodes in $\B_\ell$ have the same $\TT$-index $i$ for some $i \in [k]$. 
Hence for every node $\A_p \in \B_\ell$, $\V(\A_p) \subseteq \V_i$. 
Therefore any arbitrary set of nodes in $\B_\ell$
can be merged into an induced subgraph of $\G_i$, which also belongs to $\PP$
since $\G_i \in \PP$. 
Recall that in the Merging Phase, the algorithm tries to merge a node with
a set of leaf child-nodes connected to it via a maximum set of DCs. 
Hence a node in $\B_\ell$ whose children are all leaves is always merged with \emph{all} of its children
and turned into a leaf thereafter. 
As a result, in the Merging Phase, the algorithm traverses the branch in a bottom up manner, 
and keeps merging the leaf nodes with their parents to turn the parents into leaves.
Finally, when the algorithm reaches the top node of the branch, the whole branch
is merged into a leaf. 
\end{proof} 
\vskip 10pt 

Now we are in position to prove Lemma~\ref{lem:algo-2_terminates}. 

\vskip 10pt 
\begin{proof}[Proof of Lemma~\ref{lem:algo-2_terminates}]
As $\G \in \FF_\PP(c)$, due to Lemma~\ref{lem:phase1_terminates}, 
$\G$ passes the Splitting Phase successfully. 
It remains to show that $\G$ also passes the Merging Phase successfully. 
In fact, we show that the algorithm finds a relevant tree structure of $\G$
as soon as $\A_{r_0}$ ($i(r_0) = i_0$) is assigned to be the root node of $\TT'$. 

As shown in Lemma~\ref{lem:claim6}, when the algorithm visits a node $\A_m$, 
every nonessential branch of $\A_m$ has already been merged into a leaf node. 
The other branches of $\A_m$ are essential. By Lemma~\ref{lem:claim5}, there are
at most $c$ DCs of $\A_m$ that connect $\A_m$ to those essential branches.
A set $E_m$ that satisfies the requirements mentioned in the
Merging Phase always exists. Indeed, let $E_m$ be the set of all
nonessential DCs of $\A_m$ then
\begin{itemize}
	\item As there are at most $c$ essential DCs, $|E_m| \geq |D_m| - c$;
	\item As every branch connected to $\A_m$ via DCs in $E_m$ is nonessential, 
	        it is already merged into a leaf; Hence $C_m$ contains only leaf nodes; 
	\item Since all the branches	connected to $\A_m$ via DCs in $E_m$ are nonessential, 
	        a similar argument as in the proof of Lemma~\ref{lem:claim6} shows that the leaf child-nodes  
					of $\A_m$ in $C_m$ can be merged with $\A_m$ to produce a graph that belongs to $\PP$. 
\end{itemize}
After being merged, $\A_m$ has at most $c$ DCs. 
When the algorithm reaches the root node, $\TT'$ is turned into a relevant tree structure of $\G$. 
Thus, when $\A_{r_0}$ is chosen as the root of $\TT'$, the algorithm runs smoothly in the Merging Phase
and finds a relevant tree structure of $\G$. 
\end{proof} 

\vskip 10pt 
\begin{lemma}
\label{lem:polynomial_time_ALGO-2}
The running time of Algorithm 2 is polynomial with respect to the order of $\G$. 
\end{lemma}
\begin{proof}
Every single task in the Splitting Phase can be accomplished in polynomial time. 
Those tasks include: finding a bridge in a connected graph (see Tarjan~\cite{Tarjan74}), deciding whether a graph belongs to 
$\PP$, 
and building a tree based on the components of $\G$. 

Let examine the ``while'' loop and the ``for'' loop in the Splitting Phase. 
After each intermediate iteration in the while loop, 
as at least one component gets split into two smaller components, 
the number of components of $\G$ is increased by at least one. 
Since the vertex sets of the components are
pairwise disjoint, there are no more than $n=|\VG|$ components at any time. 
Hence, there are no more than $n$ iterations in the while loop. 
Since the number of graphs in $\Q_1$ cannot exceed $n$, 
the number of iterations in the for loop is also at most $n$. 
Therefore, the Splitting Phase finishes in polynomial time
with respect to $n$.

We now look at the running time of the Merging Phase.  
Each ``for'' loop has at most $n$ iterations and therefore does not
raise any complexity issue.  
The only task that needs an explanation is the task of finding
a maximum subset $E_m$ of DCs of $\A_m$ that satisfies
certain requirements. This task can be done by examining all
$s$-subsets of $D_m$ with $s$ runs from $|D_m|$ down to
$|D_m| - c$. There are 
\[
\sum_{s=|D_m|-c}^{|D_m|} \binom{|D_m|}{s} = \sum_{i = 0}^{c} \binom{|D_m|}{i} \leq \sum_{i=0}^c \binom{n}{i} = O(n^c)
\]
such subsets. For each subset, the verification of the two conditions specified in the algorithm can also be done in polynomial time. Therefore, the Merging Phase's running time
is polynomial with respect to $n$.  
\end{proof} 

\vskip 10pt
\begin{proof}[Proof of Theorem~\ref{thm:recognition1}]
Lemma~\ref{lem:membership_proof}, Lemma~\ref{lem:algo-2_terminates}, 
and Lemma~\ref{lem:polynomial_time_ALGO-2} qualify Algorithm 2 as
a polynomial time algorithm to recognize a member of $\FF_\PP(c)$. 
Thus Theorem~\ref{thm:recognition1} follows. 
\end{proof}
\vskip 10pt

Algorithm~2 can be adjusted, by replacing $c$ by $c\log |\VG| $, to recognize
a graph $\G$ in $\FF_\PP(c\log(\cdot))$, for any constant $c > 0$. However, according to
the proof of Lemma~\ref{lem:polynomial_time_ALGO-2}, the running time
of the algorithm in this case is roughly $O(n^{c\log n})$ ($n = |\VG|$), which is no longer polynomial in $n$.   

\section{Min-Ranks of Graphs of Small Orders}
\label{sec:small_graphs}

To aid further research on the behavior of min-ranks of graphs, 
we have carried out a computation of binary min-ranks of \emph{all}
non-isomorphic graphs of orders up to $10$. 
\begin{figure}[H]
	\centering
		\begin{tabular}{| l | l | l |}		
			\hline
		Order & Number of & Total running time\\
		& non-isomorphic graphs & \\
		 \hline
		\hline
			$1$ & $1$ & $<1$ seconds\\
			$2$ & $2$ & $<1$ seconds\\
			$3$ & $4$  & $<1$ seconds\\
			$4$ & $11$ & $<1$ seconds\\
			$5$ & $34$ & $<1$ seconds\\
			$6$ & $156$ & $<1$ seconds\\
			$7$ & $1,044$ & $<1$ seconds\\
			$8$ & $12,346$ & $25$ seconds\\
			$9$ & $274,668$ & $56$ minutes \\
			$10$ & $12,005,168$ & $4.3$ days\\
			\hline
		\end{tabular}
	\caption{Running time for finding min-ranks of graphs or small orders}
	\label{fig:NumberOfNonIsomorphicGraphsOrSmallOrders}
\end{figure}

A reduction to SAT (Satisfiability) problem \cite{ChaudhrySprintson}
provides us with an elegant method to compute the binary 
min-rank of a graph.  
We observed that while the SAT-based approach is very efficient for 
graphs having many edges, it does not perform well for simple instances, such as a graph
on $10$ vertices with no edges (min-rank $10$). 
For such naive instances, the SAT-solver that we used, Minisat \cite{EenSorensson2003}, 
was not able to terminate after hours of computation. 
This might be attributed to the fact that the SAT instances corresponding to a graph
with fewer edges contain more variables than those corresponding to 
a graphs with more edges on the same set of vertices. 

\begin{figure}[h]
\includegraphics[scale=0.15]{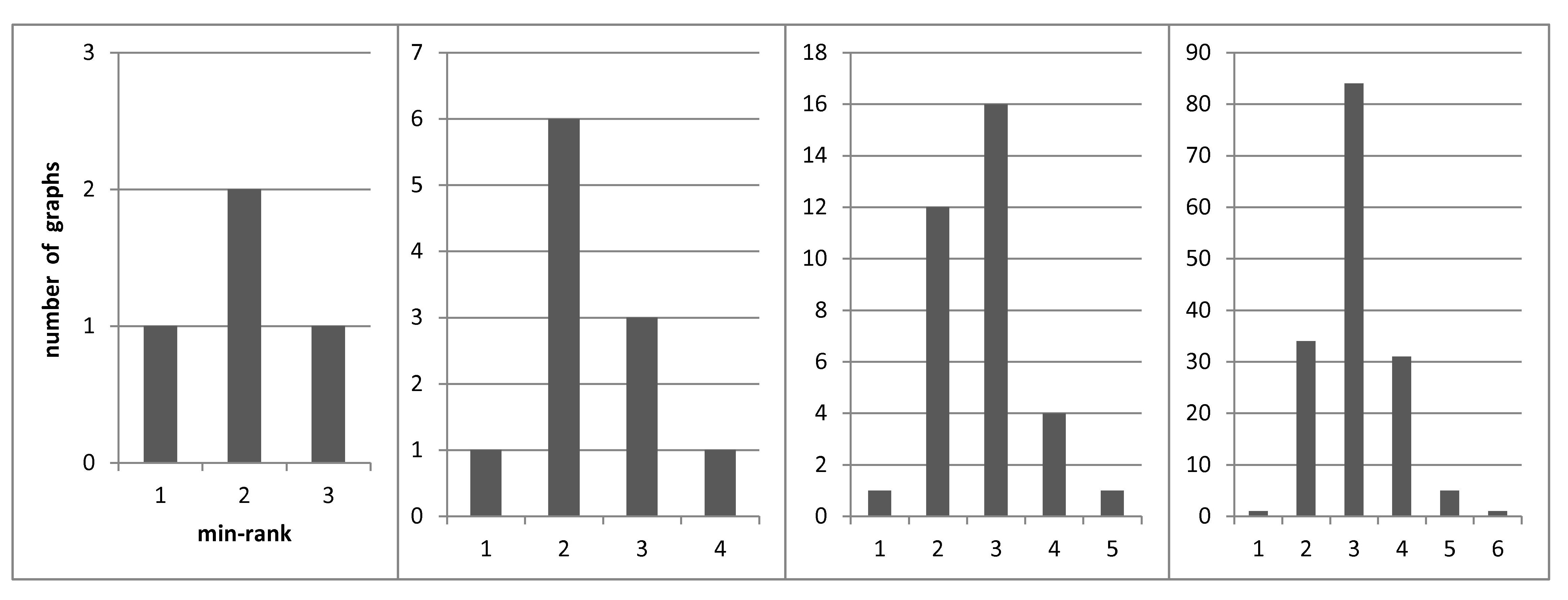}
\caption{Min-rank distributions for graphs of orders $3$--$6$}
\label{fig:distribution_3_6}
\end{figure}
 
\begin{figure}[h]
\includegraphics[scale=0.14]{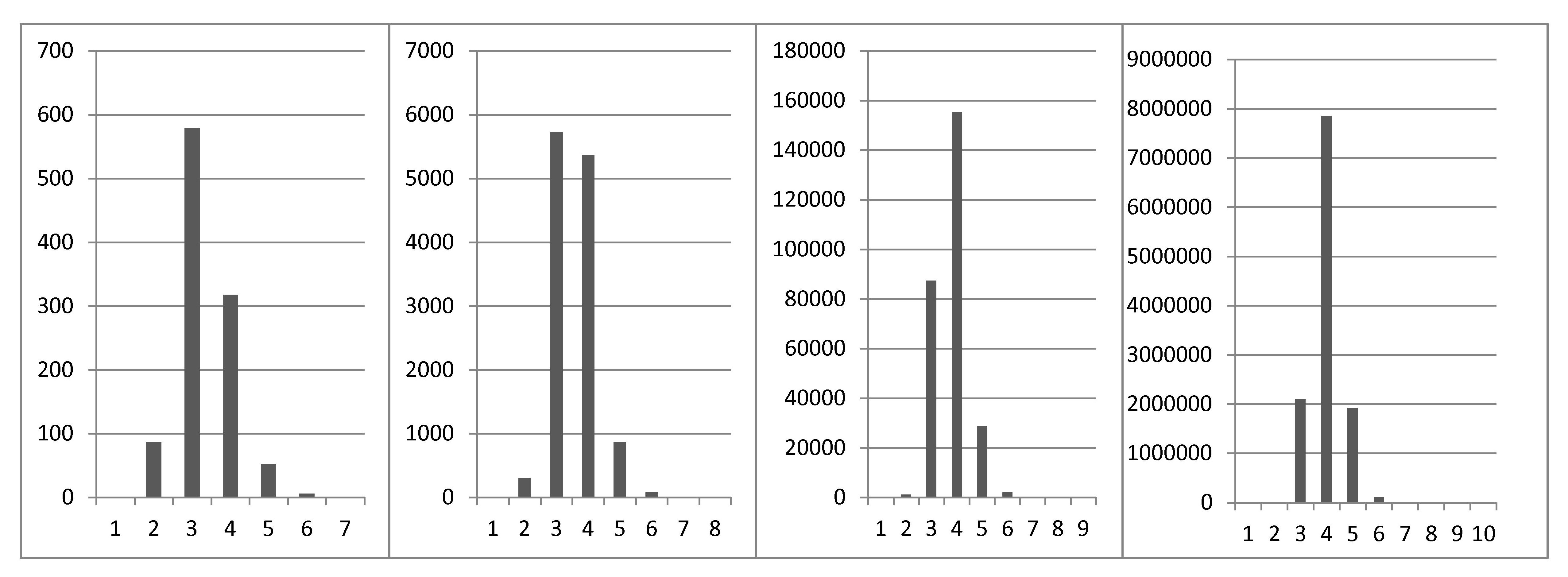}
\caption{Min-rank distributions for graphs of orders $7$--$10$}
\label{fig:distribution_7_10}
\end{figure}

To achieve our goal, we wrote a sub-program which used a Branch-and-Bound
algorithm to find min-ranks in an exhaustive manner.
When the input graph was of large size, that is, its size surpasses a given threshold, a sub-program using a SAT-solver was invoked; Otherwise, the Branch-and-Bound sub-program was used. 
We noticed that there are graphs of order $10$ that have around $21$--$22$ edges, for which
the Branch-and-Bound sub-program could find the min-ranks in less than one second, while the SAT-based sub-program could not finish computations after $3$-$4$ hours.  
For graphs of order $10$, we observed that the threshold $24$, which we actually used, did work well.
The most time-consuming task is to compute the min-ranks of all $12,005,618$
non-isomorphic graphs of order $10$. This task took more than four days
to finish. 

The charts in Figure~\ref{fig:distribution_3_6} and Figure~\ref{fig:distribution_7_10} present the distributions of min-ranks of 
non-isomorphic graphs of orders from three to ten. 
In each chart, the x-axis shows the minranks, and the y-axis shows the number of non-isomorphic
graphs that have a certain minrank.   
The minranks and the corresponding matrices that achieve the minranks of all non-isomorphic graphs of orders up to $10$ are available at \cite{small-graphs}. 
Interested reader may also visit \cite{mr} to calculate the min-rank of a graph.  

\section{Open Problems}
\label{sec:open_problems}
For future research, we would like to tackle the following open problems. \\
 
\nin {\bf Open Problem I:} Currently, in order for Algorithm~1 to work, 
we restrict ourselves to $\FF_\PP(c\log(\cdot))$, the family of graphs $\G$
having a ($\PP$) simple tree structure $\TT$ with $\mdc(\TT) \leq c\log |\VG|$ for some constant $c$. 
An intriguing question is: can we go beyond $\FF_\PP(c\log(\cdot))$?\\

\nin {\bf Open Problem II:} Find an algorithm that recognizes a member of $\FF_\PP(c \log (\cdot))$
in polynomial time, or show that there does not exist such an algorithm. \\

\nin {\bf Open Problem III:} Computation of min-ranks of graphs with $k$-multiplicity tree structures
is open for every $k \geq 2$. The $2$-multiplicity tree structure is the simplest next case to consider. 
In such a tree structure, a node can be connected to another node by at most \emph{two} edges. 
The idea of using a dynamic programming algorithm to compute min-ranks is almost the same. 
However, there are two main issues for us to tackle. 
Firstly, we need to study the effect on min-rank when an edge is removed from the graph.
In other words, we must know the relation between $\mrq(\G)$ and $\mrq(\G-e)$ for an edge $e$
of $\G$. This relation was investigated for outerplanar graphs by Berliner and Langberg~\cite[Claim~4.2, Claim~4.3]{BerlinerLangberg2011}. We need to extend their result to a new scenario. 
Secondly, as now the two nodes in the tree structure can be connected by two edges, 
a recognition algorithm for graphs with $2$-multiplicity tree structures could be more complicated
than that for graphs with simple tree structures. \\

\nin {\bf Open Problem IV:} Extending the current results to directed graphs. 

\section{Acknowledgments}
We thank Vitaly Skachek for useful comments on the draft of the paper.  
We also thank Michael Langberg for providing the preprints 
\cite{HavivLangberg2011}, \cite{BerlinerLangberg2011}.
 
\bibliographystyle{spmpsci}      
\bibliography{Polynomial_Time_Algorithm_Graphs_Tree_Structure}   

\end{document}